\documentclass[12pt,a4paper]{article}

\usepackage{bm}
\usepackage{color}
\usepackage{mathrsfs}
\usepackage{fullpage,epsf,graphicx,amstext,url}
\usepackage{amsmath,amsthm,amssymb}
\usepackage[breaklinks,colorlinks,linkcolor=black,citecolor=black,urlcolor=black]{hyperref}
\usepackage{enumerate}
\usepackage{float}
\usepackage{tikz}
\usepackage{subfigure,svg}
\usepackage{physics}
\usepackage{authblk}

\numberwithin{equation}{section}

\def\sgn{\text{sgn}}
\def\SL{\operatorname{SL}}
\def\GL{\operatorname{GL}}
\def\PSL{\operatorname{PSL}}
\def\Re{\operatorname{Re}}
\def\Im{\operatorname{Im}}
\def\N{\operatorname{N}}

\newtheorem{defi}{\bf Definition}
\newtheorem{thm}[defi]{\bf Theorem}
\newtheorem{lm}[defi]{\bf Lemma}
\newtheorem{cor}[defi]{\bf Corollary}
\newtheorem{prp}[defi]{\bf Proposition}

\newcommand{\arrowIn}{
\tikz \draw[-stealth] (-1pt,0) -- (1pt,0);
}

\numberwithin{equation}{section}

\begin{document}

\title{Geometry and distribution of roots of $\mu^2\equiv D\bmod m$ with $D \equiv 1 \bmod 4$}
\date{}
\author[$1$]{Zonglin Li\footnote{Funded by China Scholarship Council}}
\author[$2$]{Matthew Welsh\footnote{Supported by EPSRC grant EP/S024948/1}}
\affil[$1$]{School of Mathematics, University of Bristol, Bristol BS8 1UG, U.K.}
\affil[$2$]{Department of Mathematics, University of Maryland, MD 20742-5025, U.S.}

\date{28 April 2023}
\maketitle

\begin{abstract}
  We extend the geometric interpretation of the roots of the quadratic congruence $\mu^2 \equiv D \pmod m$ as the tops of certain geodesics in the hyperbolic plane to the case when $D$ is square-free and $D\equiv 1\pmod 4$.
  This allows us to establish limit laws for the fine-scale statistics of the roots using results by Marklof and the second author. 
\end{abstract}

\tableofcontents

\section{Introduction}

The recent paper \cite{misc:Marklof-Welsh} by Marklof and the second author established limit laws for the fine-scale distribution of the roots of the quadratic congruence $\mu^2\equiv D \pmod m$, ordered by the modulus $m$, where $D$ is a square-free positive integer and $D\not\equiv 1\pmod 4$.
This is achieved by relating the roots (when $D > 0$) to the tops of certain geodesics in the Poincar\'e upper half-plane $\mathbb{H} = \{ z \in \mathbb{C}: \Im z > 0\}$ and (when $D<0$) to orbits of points, under the action of the modular group. Previous studies of quadratic congruences by means of the spectral theory of automorphic forms include the works of Bykovskii \cite{jr:Bykovskii}, Duke, Friedlander and Iwaniec \cite{jr:DFI}, Hejhal \cite{jr:Hejhal} and many others (see \cite{misc:Marklof-Welsh} for further context and references).

Geodesics in $\mathbb{H}$ are either vertical lines or semi-circles centred on the real line, and in the latter case we say that a geodesic is positively oriented if it is traced from left to right.
Letting $z_{\bm{c}}$ denote the top of a positively oriented geodesic $\bm{c}$, i.e. the point on $\bm{c}$ with the largest imaginary part, the correspondence established in \cite{misc:Marklof-Welsh} is that for $D$ positive, square-free and $D \not\equiv 1 \pmod 4$, there is a finite set of geodesics $\{\bm{c}_1, \dots, \bm{c}_h\}$ such that 
\begin{equation}
  \label{eq:parametrization}
  z_{\gamma \bm{c}_l} \equiv \frac{\mu}{m} + \mathrm{i} \frac{\sqrt{D}}{m} \pmod {\Gamma_\infty}
\end{equation}
parametrizes all the roots $\mu \bmod m$ as $l$ varies over $1\leq l \leq h$ and $\gamma$ varies over the double cosets $\Gamma_\infty \backslash \SL(2, \mathbb{Z}) / \Gamma_{\bm{c}_l}$ such that $\gamma \bm{c}_l$ is positively oriented.
Here $\Gamma_{\bm{c}_l}$ is the stabilizer in $\SL(2, \mathbb{Z})$ of the geodesic $\bm{c}_l$ (acting on $\mathbb{H}$ by fractional linear transformations) and
\begin{equation}
  \label{eq:Gammainfinitydef}
  \Gamma_\infty = \left\{ \pm
    \begin{pmatrix}
      1 & k \\
      0 & 1
    \end{pmatrix}
    : k \in \mathbb{Z} \right\}
\end{equation}
is the stabilizer in $\SL(2, \mathbb{Z})$ of the boundary point $\infty \in \partial \mathbb{H}$.
Moreover, it is established that, for a fixed root $\nu \bmod n$, the subset of the roots $\mu \bmod m$ satisfying $m\equiv 0 \pmod n$ and $\mu \equiv \nu \pmod n$ is parametrized in the same way by \eqref{eq:parametrization} with $\SL(2, \mathbb{Z})$ replaced by the congruence subgroups
\begin{equation}
  \label{eq:Gamma0ndef}
  \Gamma_0(n) = \left\{
    \begin{pmatrix}
      a & b \\
      c & d
    \end{pmatrix}
    \in \SL(2, \mathbb{Z}) : c \equiv 0 \pmod n \right\}
\end{equation}
and a possibly different finite set of geodesics $\{\bm{c}_1, \dots, \bm{c}_h\}$.



In this paper, we extend this geometric realization of the roots to remove the restriction $D\not\equiv 1\pmod 4$, obtaining the following theorem.

\begin{thm}
    Let $D>0$ be a square-free integer satisfying $D\equiv 1 \pmod 4$. Given $n>0$ and $\nu \bmod n$ with $\nu^2\equiv D \pmod n$, there exists a finite set of geodesics $\{\bm{c}_1,\cdots,\bm{c}_{h}\}$, having the following properties:
    \begin{enumerate}[(i)]
        \item For any positive integer $m$ and $\mu\bmod m$ satisfying $\mu^2\equiv D \pmod m$, $m\equiv 0\pmod n$, and $\mu\equiv\nu\pmod n$, there exists a unique $k$ and a double coset $\Gamma_\infty\gamma\Gamma_{\bm{c}_k}\in\Gamma_\infty\backslash\Gamma_0(n)\slash\Gamma_{\bm{c}_k}$ with $\gamma \bm{c}_k$ positively oriented such that 
          \begin{equation}
            \label{eq:parametrization1}
            z_{\gamma\bm{c}_k}\equiv\frac{\mu}{m}+\mathrm{i}\frac{\sqrt{D}}{m} \pmod {\Gamma_\infty}.
        \end{equation}
        \item Conversely, given $k$ and a double coset $\Gamma_\infty\gamma\Gamma_{\bm{c}_k}\in\Gamma_\infty\backslash\Gamma_0(n)\slash\Gamma_{\bm{c}_k}$ with positively oriented $\gamma\bm{c}_k$, there exists a unique positive integer $m$ and $\mu\bmod m$ satisfying $\mu^2\equiv D \pmod m$, $m\equiv 0\pmod n$, and $\mu\equiv\nu\pmod n$ such that \eqref{eq:parametrization1} holds.
    \end{enumerate}
\label{thm: congruence subgroups}\end{thm}

The main difficulty in the proof of Theorem \ref{thm: congruence subgroups} is that when $D\equiv 1\pmod 4$, the quadratic order $\mathbb{Z}[\sqrt D]$ is not maximal; the ring of integers of $\mathbb{Q}(\sqrt{D})$ with $D\equiv 1\pmod 4$ is the larger order $\mathbb{Z}[\frac{1+\sqrt{D}}{2}]$.
This complicates the situation because not all ideals in $\mathbb{Z}[\sqrt{D}]$ are invertible, and so do not correspond naturally with geodesics in $\mathbb{H}$. 
We resolve this by simultaneously considering ideals (and their corresponding geodesics) from the two orders $\mathbb{Z}[\sqrt D]$ and $\mathbb{Z}[\frac{1 + \sqrt{D}}{2}]$.
In fact, ideals in the two orders correspond to two complementary subsets of roots, see Theorems \ref{thm: congruence subgroups I} and \ref{thm: congruence subgroups J} below.




\begin{figure}
\centering
\subfigure{\resizebox{0.49\textwidth}{!}{\includegraphics{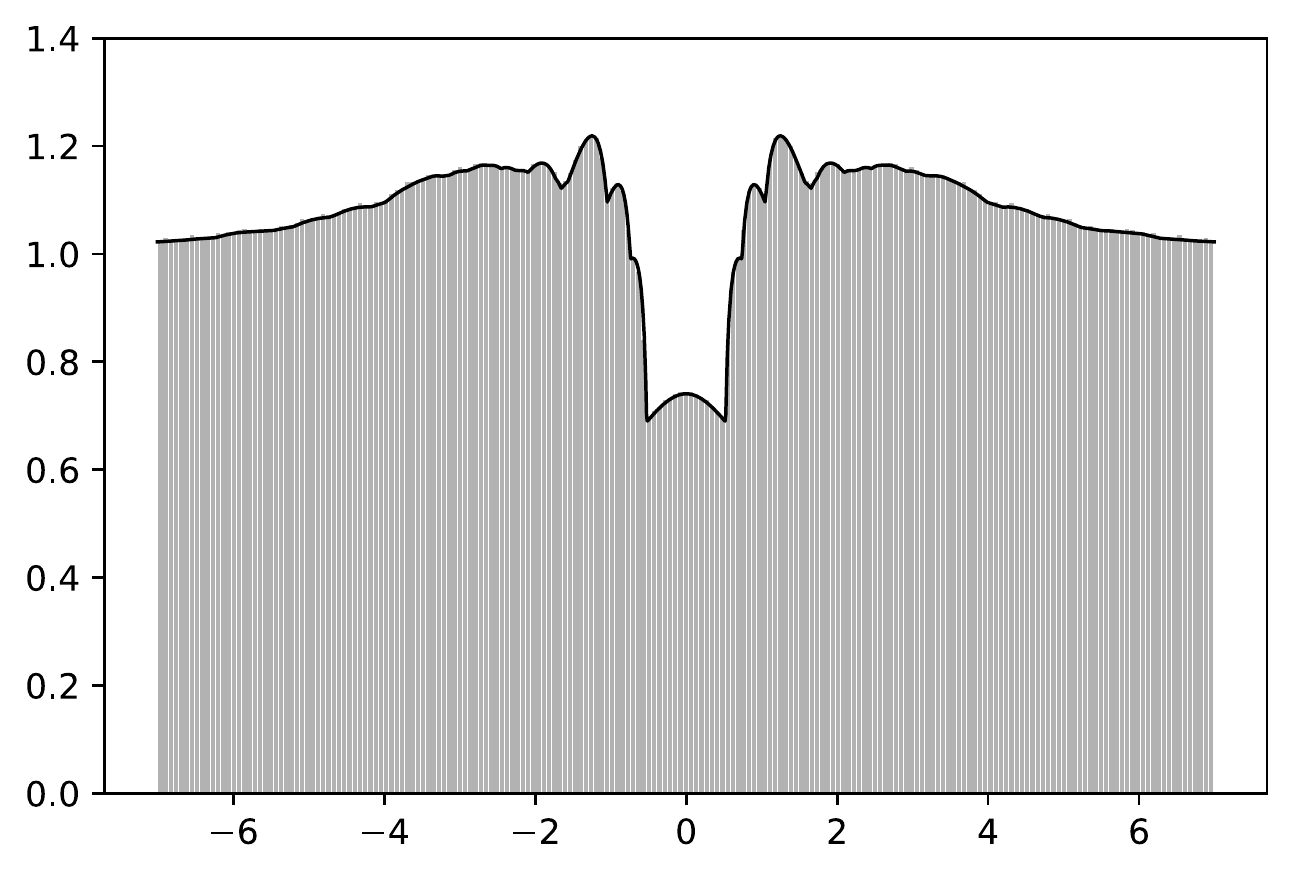}}}
\subfigure{\resizebox{0.49\textwidth}{!}{\includegraphics{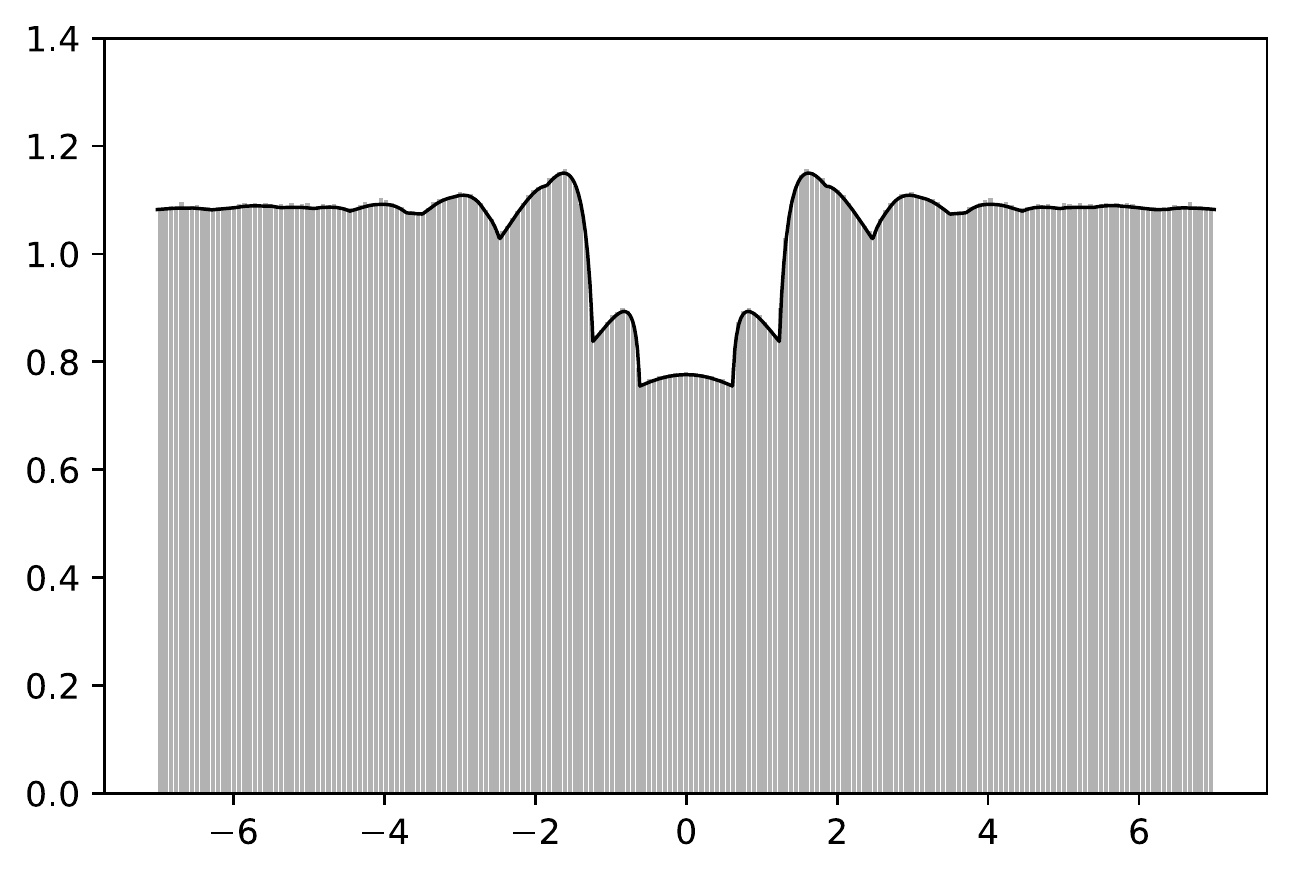}}}
\caption{Theoretical pair correlation total densities $\omega_5$ (left, black curve) and $\omega_{17}$ (right, black curve) compared to their experiments (gray histograms) with $N=10^6$.}
\label{fig:TPCD5}\end{figure}

\bigskip

Having Theorem \ref{thm: congruence subgroups}, we may apply the results of \cite{misc:Marklof-Welsh} to study the fine-scale distribution of these roots.
Specifically, we let $\{x_j\}_{j=1}^\infty$ denote the sequence of normalized roots $\frac{\mu}{m} \in \mathbb{T} = \mathbb{R} / \mathbb{Z}$ satisfying $m \equiv 0 \pmod n$ and $\mu \equiv \nu \pmod n$ for a fixed root $\nu^2 \equiv D \pmod n$.
Here we order by the modulus, with the ordering of the roots with the same modulus being inconsequential.
For example, with $D = 5$, $n =1$ and $\nu =0$, our sequence would begin
\begin{equation}
  \label{eq:seqex}
  \left\{ \frac{0}{1}, \frac{1}{2}, \frac{1}{4}, \frac{3}{4}, \frac{0}{5}, \frac{5}{10}, \frac{4}{11}, \frac{7}{11}, \dots \right\}. 
\end{equation}

We remark that it has been proved by Hooley \cite{jr:Hooley1} that these sequences (without the congruence restrictions, i.e. for $n =1$, $\nu =0$) are uniformly distributed mod $1$ if $D$ is not a perfect square.
We also remark that the uniform distribution of the subsequence with $m \equiv 0 \pmod n$ was proved by Iwaniec \cite{jr:Iwaniec}, where it was a crucial input to proving that the polynomial $x^2 + 1$, for example, is infinitely often a product of at most two primes. 

Having understood the uniform distribution of the sequence $\frac{\mu}{m}$, it is natural to investigate its pseudo-randomness properties, in particular the fine-scale statistics.
An attractive example of such a statistic is the pair correlation function $R_{N,\nu,n}^2(I)$, which for a finite interval $I\subset \mathbb{R}$ is defined by
\begin{equation}
  R_{N,\nu,n}^2(I)=\frac{1}{N}\#\!\left\{1\leq i\neq j\leq N: x_i - x_j \in N^{-1} I + \mathbb{Z} \right\}.
\end{equation}  
Theorem \ref{thm: congruence subgroups} together with theorem $4.10$ in \cite{misc:Marklof-Welsh} shows the existence of a limiting density function $\omega_{D,\nu,n}$ for the pair correlation function.

\begin{thm}
    Let $D>0$ be a square-free integer with $D\equiv 1 \pmod 4$ and $\nu \bmod n$ satisfy $\nu^2\equiv D\pmod n$. Then there exists an even and continuous function $\omega_{D,\nu,n}: \mathbb{R}\to \mathbb{R}_{\geq 0}$ such that for any bounded interval $I\subseteq\mathbb{R}$,
    \begin{equation}
        \lim_{N\to\infty} R_{N,\nu,n}^2(I)=\int_I\omega_{D,\nu,n}(t)\,dt.
    \end{equation}   
\label{thm: pair correlation density}\end{thm}

\begin{figure}
\centering
\subfigure{\resizebox{0.49\textwidth}{!}{\includegraphics{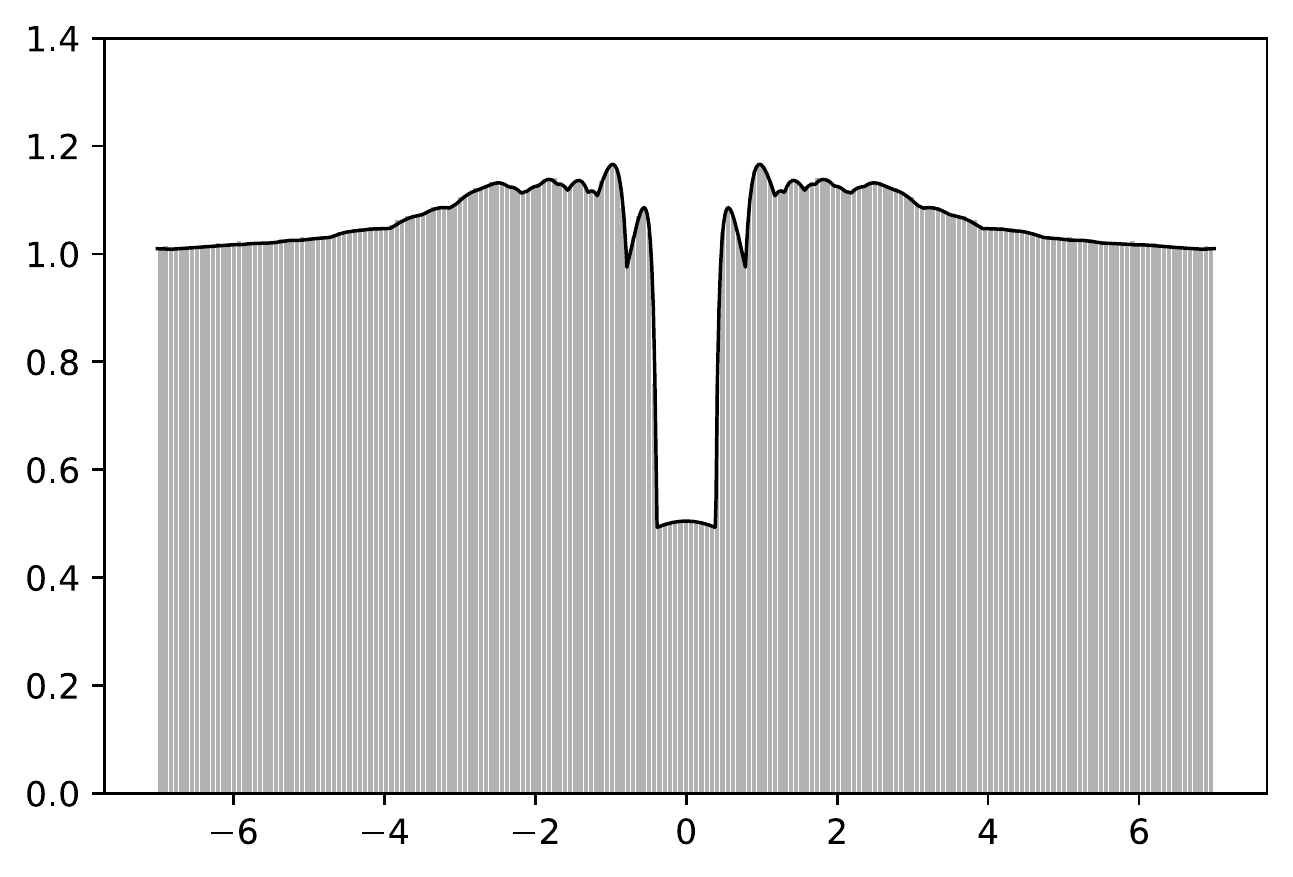}}}
\subfigure{\resizebox{0.49\textwidth}{!}{\includegraphics{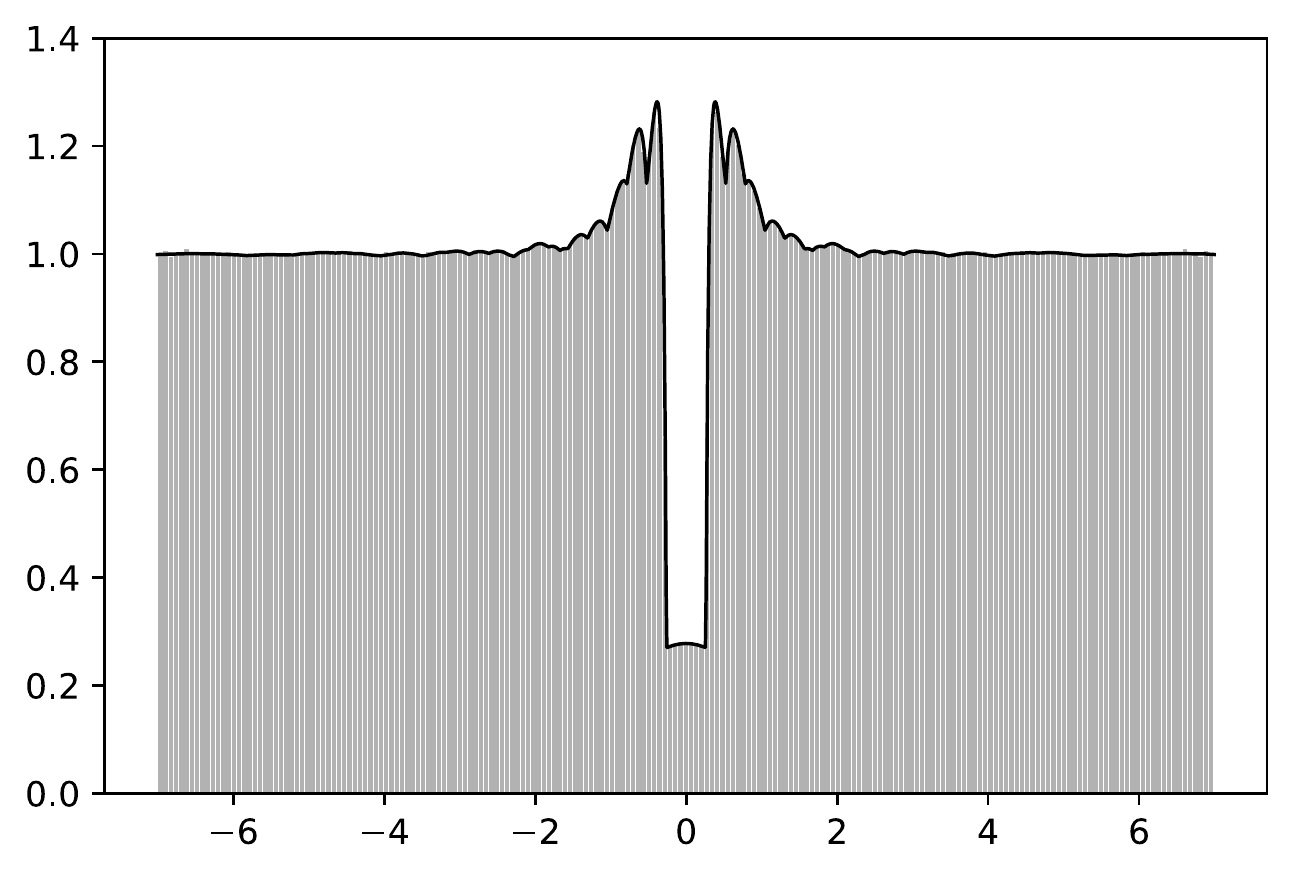}}}
\caption{Theoretical pair correlation partial densities $\omega_{5,\mathcal{O}_1}$ (left, black curve) for $m\not\equiv 2\pmod 4$ and $\omega_{5,\mathcal{O}_2}$ (right, black curve) for $m\equiv 2\pmod 4$ compared to their experiments (gray histograms) with $N=10^6$.}
\label{fig:TPCD5 Gamma}\end{figure}

As in \cite{misc:Marklof-Welsh}, the methods in fact give an explicit expression for the density functions $\omega_{D,\nu,n}$.
Moreover, we provide a somewhat simplified expression than that in \cite{misc:Marklof-Welsh}, see section \ref{sec:density}. 
The figures \ref{fig:TPCD5}, \ref{fig:TPCD5 Gamma} and \ref{fig:TPCD17 Gamma} show plots of the density functions $\omega_{D,\nu,n}$ together with numerical experiments.
For comparison, we remark that the pair correlation density for a Poisson point process would be the constant function $\omega(t) = 1$, while the normalized eigenvalues of a large random matrix and (conjecturally) the zeros of the Riemann zeta function have pair correlation density $\omega(t) = 1 - (\frac{\sin\pi t}{ \pi t})^2$. 

In Figures \ref{fig:TPCD5}, \ref{fig:TPCD5 Gamma} and \ref{fig:TPCD17 Gamma}, the pair correlation total density $\omega_D=\omega_{D,0,1}$ considers the sequence of all roots $\frac{\mu}{m}$ of $\mu^2\equiv D \pmod m$, while the pair correlation partial density $\omega_{D,\mathcal{O}_1}$ is for the subsequence of the roots $\frac{\mu}{m}$ corresponding to ideals in the ring $\mathcal{O}_1 = \mathbb{Z}[\sqrt{D}]$, while the partial density $\omega_{D,\mathcal{O}_2}$ is for the subsequence of the roots $\frac{\mu}{m}$ corresponding to ideals in the ring $\mathcal{O}_2 = \mathbb{Z}[\tfrac{1 + \sqrt D}{2}]$.
An explicit condition on the roots determining which subset they belong to is given in Lemma \ref{lm:general invertible ideal}, and in the case that $D \equiv 5 \pmod 8$, can be reduced to a congruence condition on $m$, see Lemma \ref{lm: D 5 mod 8}.

\bigskip

We can in fact prove limit distributions for a wide variety of fine-scale statistics in addition to the pair correlation.
We let $\lambda$ be a Borel probability measure on $\mathbb{T}$ and $X$ be a random variable distributed according to $\lambda$.
We then define the following random counting measure, or random point process,
\begin{equation}
    \Xi_{N,\lambda,\nu,n}=\sum_{k\in\mathbb{Z}}\sum_{j=1}^N\delta_{N(x_j-X+k)},
\end{equation}
where $\delta_y$ is a Dirac measure at $y\in\mathbb{R}$.
Given a non-negative integer $k$ and $I\subseteq\mathbb{R}$, we further denote the counting functions
\begin{equation}
    \mathcal{N}_{I,\nu,n}(x,N)=\#\!\left\{1\leq j\leq N: x_j -x \in N^{-1} I+\mathbb{Z}\right\},
\end{equation}
so that
\begin{equation}
    \mathbb{P}(\Xi_{N,\lambda,\nu,n}(I)=k)=\lambda(\{x\in\mathbb{T}:\mathcal{N}_{I,\nu,n}(x,N)=k\}).
\end{equation}
The next theorem can be proved by applying theorem $2.1$ in \cite{misc:Marklof-Welsh}.

\begin{figure}
 \centering
 \subfigure{\resizebox{0.49\textwidth}{!}{\includegraphics{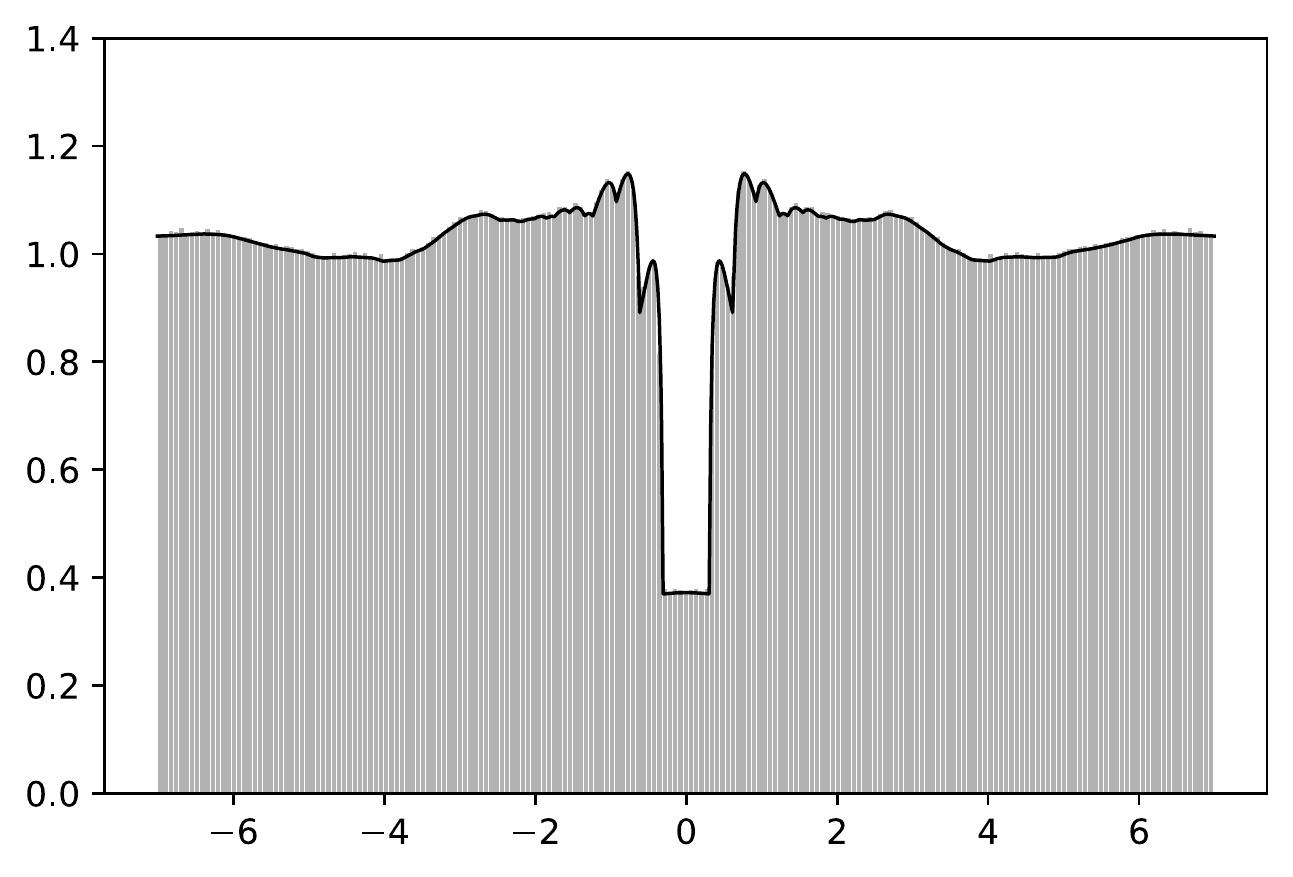}}}
 \subfigure{\resizebox{0.49\textwidth}{!}{\includegraphics{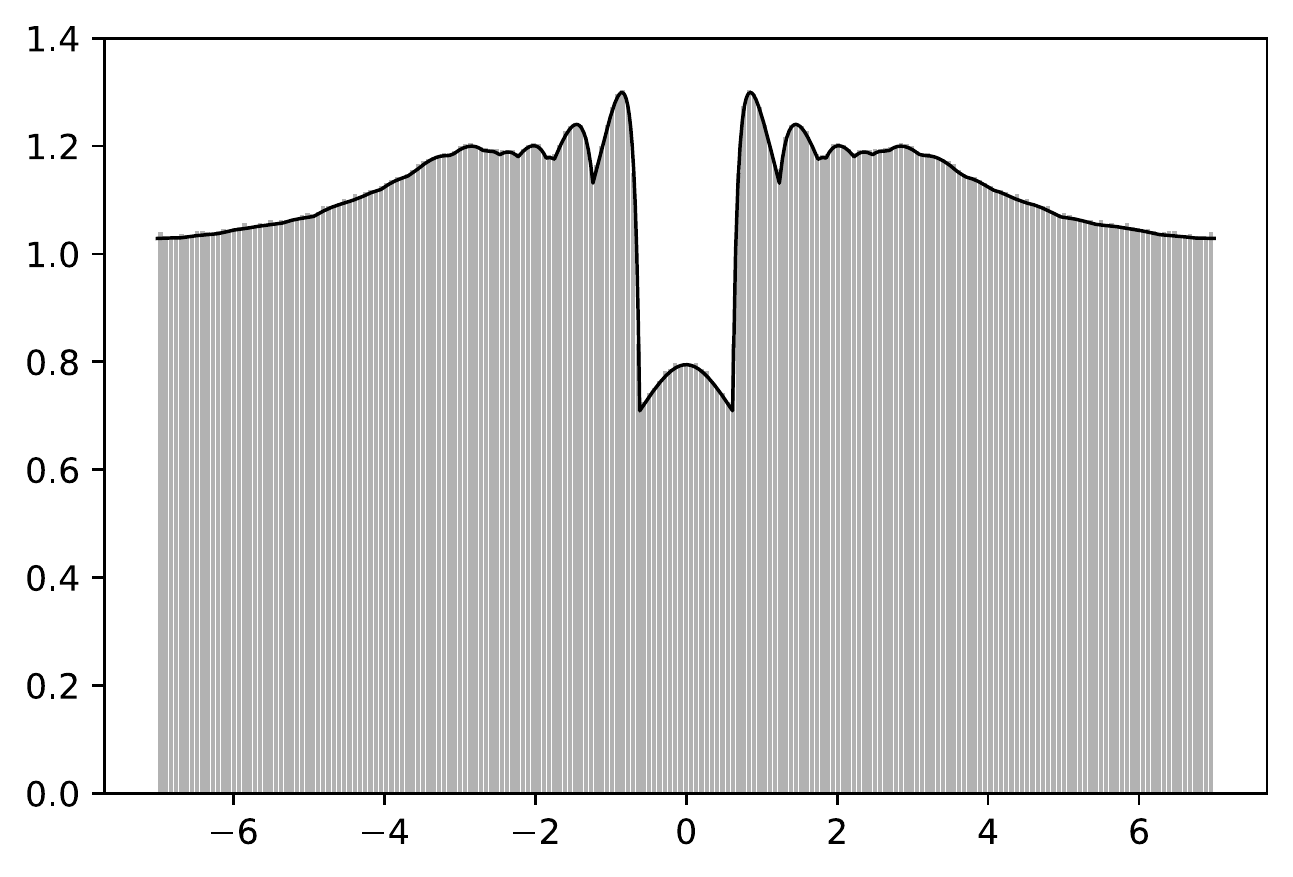}}}
 \caption{Theoretical pair correlation partial densities $\omega_{17,\mathcal{O}_1}$ (left, black curve) with either $m$ or $\frac{D-\mu^2}{m}$ odd, and $\omega_{17,\mathcal{O}_2}$ (right, black curve) with both $m$ and $\frac{D-\mu^2}{m}$ even compared to their experiments (gray histograms) with $N=10^6$.}
\label{fig:TPCD17 Gamma}\end{figure}

\begin{thm}
    Let $D>0$ be a square-free integer with $D\equiv 1 \pmod 4$ and $\nu \bmod n$ satisfy $\nu^2\equiv D\pmod n$. Then there exists a random point process $\Xi_{\nu,n}$ (depending only on $\nu, n$ and $D$) such that for any Borel probability measure $\lambda$ absolutely continuous with respect to Lebesgue measure,
    \begin{equation}
        \lim_{N\to\infty} \Xi_{N,\lambda,\nu,n} = \Xi_{\nu,n}.
    \end{equation}
    In particular, given $k_j\in\mathbb{N}$ and bounded interval $I_j\subseteq\mathbb{R}$, then
    \begin{equation}
        \lim_{N\to\infty} \lambda(\{x\in\mathbb{T}: N_{I_j,\nu,n}(x,N)=k_j, \forall\, 1\leq j \leq r\})=\mathbb{P}(\Xi_{\nu,n}(I_j)=k_j, \forall\, 1\leq j \leq r),
    \end{equation}
    which is a continuous function of the endpoints of $I_j$.
\label{thm: random point process}\end{thm}

The problem becomes simpler when considering negative $D$.
Our results still hold by considering the fine-scale distribution of real parts of hyperbolic lattice points instead of tops of geodesics, see \cite{jr:Marklof-Vinogradov, jr:Risager-Rudnick} and section \ref{sec:Dnegative} below.
Moreover, these are closely related to the study of the distribution of angles in hyperbolic lattices, see \cite{jr:Boca-Pasol-Popa, jr:Boca-Popa, jr:Kelmer-Kontorovich, jr:Risager-Sodergren}.

\bigskip

We start in section \ref{sec:geodesics} by establishing the connection between the roots $\frac{\mu}{m}$ and the tops of geodesics in the Poincar\'e upper half-plane $\mathbb{H}$.
The main result of this section is a proof of Theorem \ref{thm: congruence subgroups} in the special case $n =1$ and $\nu = 0$.
This restriction is removed, and so Theorem \ref{thm: congruence subgroups} is proved, in section \ref{sec: Congruence subgroups}, where we investigate the orbits of geodesics under the congruence subgroups \eqref{eq:Gamma0ndef}.
In section \ref{sec:density} we record some observations about the explicit expression for the pair correlation obtained in \cite{misc:Marklof-Welsh}, and we close the paper in section \ref{sec:Dnegative} with some remarks on analogous results in the case that $D$ is negative.

\bigskip
\noindent
{\bf Acknowledgement}. The authors would like to thank Jens Marklof for his valuable suggestions and encouragement.

\section{Geodesics and roots of quadratic congruences}
\label{sec:geodesics}

A convenient model for the hyperbolic plane is the Poincar\'e upper half-plane
\begin{equation}
    \mathbb{H}=\{x+\mathrm{i} y\in\mathbb{C}:y>0\}
\end{equation}
equipped with the  metric
\begin{equation}
  \label{eq:metricdef}
  \dd s^2 = \frac{1}{y^2} ( \dd x^2 + \dd y^2). 
\end{equation}
The special linear group $G = \SL(2,\mathbb{R})$, consisting of $2\times2$ matrices with real entries and determinant one, acts on $\mathbb{H}$
by M\"obius transformations
\begin{equation}
  g=\begin{pmatrix}
      a & b \\
      c & d
  \end{pmatrix}
  :z\mapsto\frac{az+b}{cz+d}.
\end{equation}
We can in fact identify the projective special linear group $\PSL(2, \mathbb{R}) = G / \{\pm I\}$ with the unit tangent bundle $\mathrm{T}_1(\mathbb{H})$ of the hyperbolic plane via
\begin{equation}
  \label{eq:GT1H}
  \pm g = \pm
  \begin{pmatrix}
    a & b \\
    c & d
  \end{pmatrix}
  \mapsto
  \left( \frac{a \mathrm{i} + b}{c \mathrm{i} + d}, \frac{\mathrm{i}}{(c \mathrm{i} + d)^2} \right),
\end{equation}
where we have expressed $\mathrm{T}_1(\mathbb{H})$ as the set of pairs of complex numbers $(z, w)$ with $z \in \mathbb{H}$ and $|w| = \Im z$ (so $w$ has length $1$ with respect to the metric \eqref{eq:metricdef}).
We will however work with $G$ directly and only emphasize the difference between $G$ and $\mathrm{T}_1(\mathbb{H})$ when there is a chance of confusion. 
We note that the  correspondence \eqref{eq:GT1H} is conveniently expressed in terms of the Iwasawa coordinates:
\begin{equation}
  \label{eq:iwasawa}
  \begin{pmatrix}
    1 & x \\
    0 & 1
  \end{pmatrix}
  \begin{pmatrix}
    y^{\frac{1}{2}} & 0 \\
    0 & y^{-\frac{1}{2}}
  \end{pmatrix}
  \begin{pmatrix}
    \cos \frac{\theta}{2} & - \sin \frac{\theta}{2} \\
    \sin \frac{\theta}{2} & \cos \frac{\theta}{2}
  \end{pmatrix}
  \mapsto \left( x + \mathrm{i} y, \mathrm{i} y \mathrm{e}^{-\mathrm{i} \theta} \right),
\end{equation}
so here $\theta$ measures the angle clockwise from vertical.

Geodesics in $\mathbb{H}$ are either vertical lines or semi-circles centred on the real line.
An oriented geodesic $\bm{c} \subset \mathbb{H}$ is parametrized by
\begin{equation}
  \label{eq:geodesic}
  \bm{c} = \left\{ g
    \begin{pmatrix}
      s & 0 \\
      0 & s^{-1} 
    \end{pmatrix}
    \mathrm{i} : s > 0 \right\}
\end{equation}
for some $g \in G$.
We let $\bm{c}_0$ denote the vertical geodesic through $\mathrm{i}$, which corresponds to $g = I$ in \eqref{eq:geodesic}.
Any geodesic can be written as $\bm{c} = g \bm{c}_0$ for some $g \in G$.

We say that a geodesic $\bm{c}$ as in \eqref{eq:geodesic} is positively oriented if the backwards endpoint $\bm{c}^{-}$ corresponding to $s \to 0$ is less than the forward endpoint $\bm{c}^+$ corresponding to $s \to \infty$, i.e. if the geodesic is traced from left to right.
In this case, we define the point $z_{\bm{c}} \in \mathbb{H}$ to be the point on $\bm{c}$ with the largest imaginary part, and we accordingly refer to it as the top of the geodesic.
We leave $z_{\bm{c}}$ undefined for vertical or negatively oriented geodesics. 

We let $\Gamma = \SL(2, \mathbb{Z})$, and we denote by $\Gamma_\infty$ the stabilizer in $\Gamma$ of the boundary point $\infty \in \partial \mathbb{H}$.
Explicitly,
\begin{equation}
  \Gamma_\infty=\Gamma\cap\left\{\pm
    \begin{pmatrix}
        1 & x \\
        0 & 1
    \end{pmatrix}
    :x\in\mathbb{R}\right\}=\left\{\pm
    \begin{pmatrix}
        1 & n \\
        0 & 1
    \end{pmatrix}
    :n\in\mathbb{Z}\right\}.
\end{equation}
For a geodesic $\bm{c} = g\bm{c}_0$, we let $\Gamma_{\bm{c}}$ be the stabilizer of $\bm{c}$ in $\Gamma$, so
\begin{equation}
  \label{eq:Gammacdef}
  \Gamma_{\bm{c}} = \Gamma \cap \left\{ \pm g
    \begin{pmatrix}
      s & 0 \\
      0 & s^{-1}
    \end{pmatrix}
    g^{-1} : s > 0 \right\}.
\end{equation}
If $\bm{c}$ projects to a closed geodesic on the surface $\Gamma \backslash \mathbb{H}$, then $\Gamma_{\bm{c}}$ is infinite cyclic, up to sign: 
\begin{equation}
  \label{eq:closedstablizer}
  \Gamma_{\bm{c}} = \left\{ \pm g
    \begin{pmatrix}
      \varepsilon^k & 0 \\
      0 & \varepsilon^{-k}
    \end{pmatrix}
    g^{-1} : k \in \mathbb{Z} \right\},
\end{equation}
for some $\varepsilon > 1$, and the length of $\bm{c}$ in $\Gamma \backslash \mathbb{H}$ is then $2 \log \varepsilon$. 


\bigskip


\subsection{Roots and ideals}

We recall that since $D \equiv 1 \pmod 4$, the quadratic order $\mathbb{Z}[\sqrt{D}]$ is not maximal, i.e. it is not the ring of integers of $\mathbb{Q}(\sqrt{D})$.
As such, not all ideals in $\mathbb{Z}[\sqrt{D}]$ are invertible.
The following lemma gives a criterion to determine if an ideal is invertible.
Here $(\alpha_1, \dots, \alpha_k)$ denotes the ideal of a ring (which will be made clear by the context)  generated by the ring elements $\alpha_j$.

\begin{lm}
  Let $D>0$ be a square-free integer satisfying $D\equiv 1 \pmod 4$ and $\mu \bmod m$ satisfy $\mu^2 \equiv D \pmod m$.
  Then the ideal $I=(m,\sqrt{D}+\mu)\subseteq\mathbb{Z}[\sqrt{D}]$ is invertible if and only if either $m$ or $\frac{D-\mu^2}{m}$ is odd.
  Moreover, in this case the inverse ideal of $I$ is given by $I^{-1}=\frac{1}{m}(m,\sqrt{D}-\mu)$. 
  \label{lm:general invertible ideal}
\end{lm}

Before the proof we make the following remarks:
First, the parity of $\frac{D-\mu^2}{m}$ is well-defined only if $m$ is even, so the phrase ``either $m$ or $\frac{D-\mu^2}{m}$ is odd'' would be more correctly stated as ``either $m$ is odd or $m$ is even and $\frac{D - \mu^2}{m}$ is odd.''
Second, the proof of Proposition \ref{prp:general m not equiv 2 mod 4} below shows that any ideal $I \subseteq \mathbb{Z}[\sqrt{D}]$ can be expressed as a rational integer multiple of an ideal $(m, \sqrt{D} + \mu)$ as in the lemma.

\begin{proof}
  Let $\overline{I}$ denote the ideal $(m,\sqrt{D}-\mu) \subseteq \mathbb{Z}[\sqrt{D}]$.
  We have
\begin{equation}
\begin{aligned}
   I \overline{I}&=(m^2,m(\sqrt{D}+\mu),m(\sqrt{D}-\mu),D-\mu^2)\\
   &=m(m,\sqrt{D}+\mu,\sqrt{D}-\mu,(D-\mu^2)/m)\\
   &=m(m,2\mu,\sqrt{D}-\mu,(D-\mu^2)/m),
\end{aligned}
\label{eq:II-bar'}\end{equation}
so $m(m,2\mu,\frac{D-\mu^2}{m})\subseteq I \overline{I}
\subseteq m\mathbb{Z}[\sqrt{D}]$.

Now suppose that $p\neq2$ divides all of $m, 2\mu$ and $\frac{D-\mu^2}{m}$.
Then $p^2$ divides $D=m\frac{D-\mu^2}{m}+\mu^2$, which contradicts $D$ being square-free.
It follows that if either $m$ or $\frac{D-\mu^2}{m}$ is odd, then $\gcd(m,2\mu,\frac{D-\mu^2}{m})=1$, and so $(m,2\mu,\frac{D-\mu^2}{m})=\mathbb{Z}[\sqrt{D}]$.
Therefore $I\overline{I}=m\mathbb{Z}[\sqrt{D}]$, i.e. $I^{-1}=\frac{1}{m}\overline{I}$.

Conversely, we assume $m$ and $\frac{D - \mu^2}{m}$ are even.
Dividing \eqref{eq:II-bar'} by $m$, we obtain
\begin{equation}
\begin{aligned}
    \frac{1}{m}I \overline{I}&=(m,2\mu,\mu+\sqrt{D},(D-\mu^2)/m)\\
    &=(2,\mu+\sqrt{D})=(2,1+\sqrt{D})     
\end{aligned}
\end{equation}
as $\mu$ is odd.
Letting $I_0 = (2, 1 + \sqrt{D})$,
we have
\begin{equation}
\begin{aligned}
    I_0^2&=(2,1+\sqrt{D})^2=(4,2+2\sqrt{D},1+D+2\sqrt{D})\\
    &=2(2,1+\sqrt{D},(D-1)/2)=2 I_0
\end{aligned}
\end{equation}
since $\frac{D-1}{2}$ is even.
As $I_0 \neq (2)$, it is therefore not invertible, so either $I$ or $\overline{I}$ is not invertible, and it follows that neither are as they are conjugate. 
\end{proof}

We now turn to the correspondence between roots and ideals.

\begin{prp}
  Let $D\equiv 1 \pmod 4$ and let $I\subseteq\mathbb{Z}[\sqrt{D}]$ be an ideal with no rational integer divisors greater than $1$.
  Then $I$ has a $\mathbb{Z}$-basis $\{\beta_1, \beta_2\}$, unique modulo the action of $\Gamma_\infty$, with the form
    \begin{equation}
        \begin{pmatrix}
            \beta_1 \\
            \beta_2
        \end{pmatrix}
        =
        \begin{pmatrix}
            1 & \mu \\
            0 & m \\
        \end{pmatrix}
        \begin{pmatrix}
            \sqrt{D}\\
            1
        \end{pmatrix},
    \label{eq: beta basis}\end{equation}
  where $\mu^2\equiv D \pmod m$.
  Moreover, $I$ is invertible if and only if either $m$ or $\frac{D-\mu^2}{m}$ is odd. 
    
  Conversely, given $\mu \bmod m$ such that $\mu^2\equiv D\pmod m$,
  the sublattice of $\mathbb{Z}[\sqrt{D}]$ with $\mathbb{Z}$-basis $\{\beta_1, \beta_2\}$ given above is an ideal in $\mathbb{Z}[\sqrt{D}]$, and this ideal is invertible if and only if either $m$ or $\frac{D - \mu^2}{m}$ is odd. 
\label{prp:general m not equiv 2 mod 4}\end{prp}

\begin{proof}
Any lattice $I\subseteq\mathbb{Z}[\sqrt{D}]$ has a unique basis $\{\beta_1, \beta_2\}$ in the Hermite normal form
\begin{equation}
    \begin{pmatrix}
        \beta_1 \\
        \beta_2
    \end{pmatrix}
    =
    \begin{pmatrix}
        b_1 & b_2 \\
        0 & b_3 \\
    \end{pmatrix}
    \begin{pmatrix}
        \sqrt{D}\\
        1
    \end{pmatrix},
\end{equation}
where $b_1>0$ and $b_2>b_3 \geq 0$.
Left multiplying a matrix $\gamma\in\Gamma_\infty$ on the basis is equivalent to adding multiples of $b_3$ to $b_2$, so considering the basis modulo $\Gamma_\infty$ is the same as considering $b_2 \bmod {b_3}$.

Now, the sublattice $I$ is an ideal if and only if $\sqrt{D}I\subseteq I$, as the quadratic order $\mathbb{Z}[\sqrt{D}]$ is generated by $\sqrt{D}$. The basis of $\sqrt{D}I$ can be expressed by
\begin{equation}
\begin{aligned}
    \begin{pmatrix}
        \sqrt{D}\beta_1 \\
        \sqrt{D}\beta_2
    \end{pmatrix}
    &=
    \begin{pmatrix}
        b_1 & b_2 \\
        0 & b_3 \\
    \end{pmatrix}
    \begin{pmatrix}
        0 & D\\
        1 & 0        
    \end{pmatrix}
    \begin{pmatrix}
        b_1 & b_2\\
        0 & b_3        
    \end{pmatrix}^{-1}
    \begin{pmatrix}
        \beta_1 \\
        \beta_2
    \end{pmatrix}\\
    &=\frac{1}{b_1b_3}
    \begin{pmatrix}
        b_2b_3 & -b_2^2+Db_1^2\\
        b_3^2 & -b_2b_3
    \end{pmatrix}
    \begin{pmatrix}
        \beta_1 \\
        \beta_2
    \end{pmatrix}.
\end{aligned}
\end{equation}
Hence $I$ is an ideal if and only if the entries of the matrix in the final step are all integers.
It is clear that we need $b_1\mid b_2$ and $b_1\mid b_3$, so let $b_2=\mu b_1$ and $b_3=mb_1$.
Moreover, $I$ not having rational integer divisors is equivalent to $b_1$, which we now assume. 
From the top right entry, we obtain the constraint $\mu^2\equiv D\pmod m$.
Finally, with the assumption that $I=(m,\sqrt{D}+\mu)$ is invertible in $\mathbb{Z}[\sqrt{D}]$, we have that either $m$ or $\frac{D-\mu^2}{m}$ is odd by Lemma \ref{lm:general invertible ideal}.

Conversely, suppose that $\mu \bmod m$ satisfies $\mu^2\equiv D\pmod m$.
Then \eqref{eq: beta basis} demonstrates that given the basis in the form above, the sublattice $I$ satisfies $\sqrt{D}I\subseteq I$ as required.
Moreover, if either $m$ or $\frac{D - \mu^2}{m}$ is odd, Lemma \ref{lm:general invertible ideal} implies that $I$ is invertible.
\end{proof}

We now turn to the roots that do not satisfy the invertibility condition of Lemma \ref{lm:general invertible ideal}.
As it happens, these roots naturally correspond to ideals in the maximal order $\mathbb{Z}[\tfrac{1 + \sqrt{D}}{2}]$, where all ideals are invertible.

\begin{prp}
  Let $D\equiv 1 \pmod 4$ and let  $J\subseteq\mathbb{Z}[\frac{1+\sqrt{D}}{2}]$ be an ideal with no rational integer divisors greater than $1$.
  Then $J$ has a $\mathbb{Z}$-basis $\{\delta_1, \delta_2\}$, unique modulo the action of $\Gamma_\infty$, with the form
    \begin{equation}
        \begin{pmatrix}
            \delta_1 \\
            \delta_2
        \end{pmatrix}
        =
        \begin{pmatrix}
            1 & \frac{\mu-1}{2} \\
            0 & \frac{m}{2} \\
        \end{pmatrix}
        \begin{pmatrix}
            \frac{1+\sqrt{D}}{2}\\
            1
        \end{pmatrix},
    \label{eq: delta basis}\end{equation}
    where $\mu^2\equiv D \pmod m$ and both $m$ and $\frac{D-\mu^2}{m}$ are even.
    
    Conversely, given $\mu \bmod m$ such that $\mu^2\equiv D\pmod m$ and both $m$ and $\frac{D-\mu^2}{m}$ are even, the sublattice of $\mathbb{Z}[\frac{1+\sqrt{D}}{2}]$ with $\mathbb{Z}$-basis $\{\delta_1, \delta_2\}$ given above is an ideal in $\mathbb{Z}[\frac{1+\sqrt{D}}{2}]$.
\label{prp:general m equiv 2 mod 4}\end{prp}

\begin{proof}
As above, any lattice $J\subseteq\mathbb{Z}[\frac{1+\sqrt{D}}{2}]$ has a unique basis in Hermite normal form:
\begin{equation}
    \begin{pmatrix}
        \delta_1 \\
        \delta_2
    \end{pmatrix}
    =
    \begin{pmatrix}
        b_1 & b_2 \\
        0 & b_3 \\
    \end{pmatrix}
    \begin{pmatrix}
        \frac{1+\sqrt{D}}{2}\\
        1
    \end{pmatrix},
\end{equation}
where $b_1>0$ and $b_2>b_3 \geq 0$.
The sublattice $J$ is an ideal if and only if $\frac{1+\sqrt{D}}{2}J\subseteq J$.
The basis of $\frac{1+\sqrt{D}}{2}J$ can be expressed by
\begin{equation}
\begin{aligned}
    \begin{pmatrix}
        \frac{1+\sqrt{D}}{2}\delta_1 \\
        \frac{1+\sqrt{D}}{2}\delta_2
    \end{pmatrix}
    &=
    \begin{pmatrix}
        b_1 & b_2 \\
        0 & b_3 \\
    \end{pmatrix}
    \begin{pmatrix}
        1 & \frac{D-1}{4}\\
        1 & 0        
    \end{pmatrix}
    \begin{pmatrix}
        b_1 & b_2\\
        0 & b_3        
    \end{pmatrix}^{-1}
    \begin{pmatrix}
        \delta_1 \\
        \delta_2
    \end{pmatrix}\\
    &=\frac{1}{b_1b_3}
    \begin{pmatrix}
        (b_1+b_2)b_3 & -(b_1+b_2)b_2+\frac{D-1}{4}b_1^2\\
        b_3^2 & -b_2b_3
    \end{pmatrix}
    \begin{pmatrix}
        \delta_1 \\
        \delta_2
    \end{pmatrix}.
\end{aligned}
\end{equation}
As in the proof of Proposition \ref{prp:general m not equiv 2 mod 4}, we need $b_1$ to divide both $b_2$ and $b_3$ for the matrix to have integer entries, so we set $b_2=\nu b_1$ and $b_3=nb_1$, and as before we assume $b_1 =1$.

From the top right entry, we obtain the quadratic congruence $\nu^2+\nu\equiv\frac{D-1}{4}\pmod n$, which is equivalent to $(2\nu+1)^2\equiv D \pmod 4n$.
Let $\mu=2\nu+1$ and $m=2n$, and note that $\frac{D-\mu^2}{m}$ is even and $\mu^2\equiv D\pmod m$.

For the converse part, suppose that $\mu \bmod m$ satisfies $\mu^2\equiv D\pmod m$ and both $m=2n$ and $\frac{D-\mu^2}{m}=2k$ are even.
Further, we set $\mu=2\nu+1$ as $m$ even implies $\mu$ is odd.
We have $4\nu^2+4\nu+1\equiv D \pmod {4n}$, so $\nu^2+\nu\equiv \frac{D-1}{4} \pmod n$.
From \eqref{eq: delta basis} it then follows that the sublattice $J$ with basis $\{\delta_1, \delta_2\}$ as above satisfies $\frac{1+\sqrt{D}}{2}J\subseteq J$ as required.
\end{proof}

We record the following lemma, whose proof is obvious, that gives a simple characterization of invertible ideals in the special case when $D \equiv 5 \pmod 8$.

\begin{lm}
  Let $D\equiv 5 \pmod 8$ and $\mu \bmod m$ satisfy $\mu^2\equiv D\pmod m$. Then $m$ is even and $\frac{D-\mu^2}{m}$ is odd if and only if $m\equiv 0\pmod 4$.
  In particular, the ideal $I = (m, \mu + \sqrt{D}) \subseteq \mathbb{Z}[\sqrt{D}]$ is invertible if and only if $m\not\equiv 2\pmod 4$.
\label{lm: D 5 mod 8}\end{lm}




As $D\equiv 1 \pmod 4$ and $\mathbb{Z}[\frac{1+\sqrt{D}}{2}]$ is the ring of integers of the number field $\mathbb{Q}(\sqrt{D})$, we may define the narrow class group of $\mathbb{Z}[\frac{1+\sqrt{D}}{2}]$ in the usual way: the group of fractional ideals of $\mathbb{Z}[\tfrac{1 + \sqrt{D}}{2}]$ modulo ideals of the form $\zeta \mathbb{Z}[\tfrac{1 + \sqrt{D}}{2}]$ with $\zeta \in \mathbb{Q}(\sqrt{D})$ totally positive.
We recall that $\zeta = a + b \sqrt{D} \in \mathbb{Q}(\sqrt{D})$ is said to be totally positive if both $a + b \sqrt{D} > 0$ and $a - b\sqrt{D} > 0$.

We define the narrow class group of $\mathbb{Z}[\sqrt{D}]$ in the same way, with the added restriction that the ideals are invertible.
That is, the narrow class group of $\mathbb{Z}[\sqrt{D}]$ is the group of invertible, fractional ideals of $\mathbb{Z}[\sqrt{D}]$ modulo ideals of the form $\xi \mathbb{Z}[\sqrt{D}]$ for some $\xi \in \mathbb{Q}(\sqrt{D})$ totally positive.
We let the  narrow class numbers $h_1^+(D)$, respectively  $h_2^+(D)$, denote the order of the narrow class group in $\mathbb{Z}[\sqrt{D}]$, respectively $\mathbb{Z}[\tfrac{1 + \sqrt{D}}{2}]$.

We now use representatives of the narrow class groups to find new bases for invertible ideals $I\subseteq\mathbb{Z}[\sqrt{D}]$ and ideals $J\subseteq\mathbb{Z}[\frac{1+\sqrt{D}}{2}]$.
We fix representatives $I_k$, $1\leq k \leq h_1^+(D)$ of the narrow class group of $\mathbb{Z}[\sqrt{D}]$, so that given any invertible ideal $I\subseteq\mathbb{Z}[\sqrt{D}]$, there exists a totally positive $\xi\in I_k^{-1}$ such that $I=\xi I_k$.

To conveniently write a basis for $I$ using this equality, we embed $\mathbb{Q}(\sqrt{D})$ inside $\mathbb{R}^2$ via
\begin{equation*}
  a+b\sqrt{D}\mapsto
  (a+b\sqrt{D},a-b\sqrt{D}),
\end{equation*}
so that $\mathbb{Z}[\sqrt{D}]$ and $\mathbb{Z}[\tfrac{1 + \sqrt{D}}{2}]$ are lattices in $\mathbb{R}^2$.
For $\xi=a+b\sqrt{D}\in \mathbb{Q}(\sqrt{D})$, we use the notation $\overline{\xi}=a-b\sqrt{D}$ to represent its conjugate.

According to this embedding, given a $\mathbb{Z}$-basis $\{\beta_{k1},\beta_{k2}\}$ of $I_k$, we substitute
\begin{equation}
    \begin{pmatrix}
        \beta_{k1} \\
        \beta_{k2}
    \end{pmatrix}
    \rightarrow \mathfrak{B}_k=
    \begin{pmatrix}
        \beta_{k1} & \overline{\beta}_{k1}\\
        \beta_{k2} & \overline{\beta}_{k2}
      \end{pmatrix}
      ,
\end{equation}
and we may assume that $\det \mathfrak{B}_k > 0$.
Similarly, we write our basis for $\mathbb{Z}[\sqrt{D}]$ as
\begin{equation}
    \begin{pmatrix}
        \sqrt{D}\\
        1
    \end{pmatrix}
    \rightarrow
    \begin{pmatrix}
    \sqrt{D} & -\sqrt{D}\\
    1 & 1
    \end{pmatrix}.
\end{equation}
In this notation, we have that the basis $\{\xi\beta_{k1},\xi\beta_{k2}\}$ of $\xi I_k$ corresponds to the matrix
\begin{equation}
  \label{eq:xiIkbasis}
    \begin{pmatrix}
        \xi\beta_{k1} & \overline{\xi}\overline{\beta}_{k1}\\
        \xi\beta_{k2} & \overline{\xi}\overline{\beta}_{k2}
    \end{pmatrix}
    =\mathfrak{B}_k
    \begin{pmatrix}
        \xi & 0 \\
        0 & \overline{\xi}
    \end{pmatrix}.
  \end{equation}

Let $\varepsilon_1>1$ be the generator of the group of totally positive units of $\mathbb{Z}[\sqrt{D}]$, and we define
\begin{equation}
\begin{aligned}
    \Gamma_{1;k}&=\left\{\pm\mathfrak{B}_k\begin{pmatrix}
        \varepsilon_1^s & 0\\
        0 & \varepsilon_1^{-s}
    \end{pmatrix}\mathfrak{B}_k^{-1}:s\in\mathbb{Z}\right\}\\
    &=\SL(2,\mathbb{Z})\cap\left\{\pm\mathfrak{B}_k
      \begin{pmatrix}
        t & 0 \\
        0 & t^{-1}
      \end{pmatrix}
\mathfrak{B}_k^{-1}:t>0\right\},
\end{aligned}    
\end{equation}
since $I_k=\varepsilon I_k$ if and only if $\varepsilon$ is a unit of $\mathbb{Z}[\sqrt{D}]$.

\begin{prp}
  Let $D\equiv 1\pmod 4$ and $\mu \bmod m$ satisfy $\mu^2\equiv D \pmod m$.
  If either $m$ or $\frac{D-\mu^2}{m}$ is odd, then there exists a unique $1\leq k \leq h_1^+(D)$ and a unique double coset $\Gamma_\infty\gamma\Gamma_{1;k}\in\Gamma_\infty\backslash\SL(2,\mathbb{Z})\slash\Gamma_{1;k}$ such that 
    \begin{equation}
        \begin{pmatrix}
            1 & \mu \\
            0 & m
        \end{pmatrix}
        \begin{pmatrix}
            \sqrt{D} & -\sqrt{D}\\
            1 & 1
        \end{pmatrix}
        =\gamma\mathfrak{B}_k
        \begin{pmatrix}
            \xi & 0 \\
            0 & \overline{\xi}
        \end{pmatrix}
    \label{eq: gamma B_k}\end{equation}
    for some totally positive $\xi$.
    
    Conversely, if for $\Gamma_\infty\gamma\Gamma_{1;k}\in\Gamma_\infty\backslash\SL(2,\mathbb{Z})\slash\Gamma_{1;k}$ there exists  positive real numbers $\xi_1$, $\xi_2$ such that
    \begin{equation}
      \label{eq:rationalintersections}
        \gamma\mathfrak{B}_k
        \begin{pmatrix}
        \xi_1 & 0 \\
        0 & \xi_2
        \end{pmatrix}
        =
        \begin{pmatrix}
        1 & * \\
        0 & * \\
        \end{pmatrix}
        \begin{pmatrix}
            \sqrt{D} & -\sqrt{D}\\
            1 & 1
        \end{pmatrix},
    \end{equation}
    then in fact $\xi_1 = \xi$ and $\xi_2 = \overline{\xi}$ for some totally positive $\xi \in I_k^{-1}$ and we have
    \begin{equation}
        \gamma\mathfrak{B}_k
        \begin{pmatrix}
        \xi & 0 \\
        0 & \overline{\xi}
        \end{pmatrix}
        =
        \begin{pmatrix}
        1 & \mu \\
        0 & m
        \end{pmatrix}
        \begin{pmatrix}
            \sqrt{D} & -\sqrt{D}\\
            1 & 1
        \end{pmatrix},
    \end{equation}
    where $\mu \bmod m$ satisfies $\mu^2\equiv D\pmod m$ and either $m$ or $\frac{D-\mu^2}{m}$ is odd.
    \label{prp: B_k}
\end{prp}

\begin{proof}
  By Proposition \ref{prp:general m not equiv 2 mod 4}, the left side of \eqref{eq: gamma B_k} gives a basis for an invertible ideal $I \subseteq \mathbb{Z}[\sqrt{D}]$.
  Writing $I = \xi I_k$ for some $k$ and totally positive $\xi \in \mathbb{Q}(\sqrt{D})$, we obtain another basis for $I$ as in \eqref{eq:xiIkbasis}.
  These bases must be related by some $\gamma \in \GL(2, \mathbb{Z})$, giving \eqref{eq: gamma B_k}. Furthermore, by comparing the sign of determinants of both sides, we in fact have $\gamma \in \Gamma = \SL(2, \mathbb{Z})$.


  We note that considering $\mu$ as a residue class modulo $m$ is the same as considering the coset $\gamma$ in $\Gamma_\infty \backslash \Gamma$.
  Moreover, replacing $\xi$ up to multiplication by the group of totally positive units in $\mathbb{Z}[\sqrt{D}]$ amounts to considering the coset $\gamma\Gamma_{1;k}$.
  It follows that the double coset $\Gamma_\infty\gamma\Gamma_{1;k}$ is uniquely determined.

Conversely, suppose that there exist positive real numbers $\xi_1$, $\xi_2$ so that \eqref{eq:rationalintersections} holds.
We let $\{\beta'_{k1},\beta'_{k2}\}$ be a basis of $I_k^{-1}$ and observe that via the embedding $\mathbb{Q} \to \mathbb{R}^2$, we may express $(\xi_1, \xi_2) =c_1\beta'_{k1}+c_2\beta'_{k2}$ for some real numbers $c_1,c_2\in\mathbb{R}$.
The claim that $\xi_1 = \xi$ and $\xi_2 = \overline{\xi}$ for some totally positive $\xi \in I_k^{-1}$ is equivalent to $c_1, c_2$ being integers.

As $I_k^{-1}I_k=\mathbb{Z}[\sqrt{D}]$, we define integers $a_{kij}$ and $b_{kij}$ by
\begin{equation}
  \label{eq:abkijdef}
  \beta'_{ki}\beta_{kj}=a_{kij}+b_{kij}\sqrt{D}.
\end{equation}
This then yields
\begin{equation}
\mathfrak{B}_k
    \begin{pmatrix}
        \xi_1 & 0 \\
        0 & \xi_2
    \end{pmatrix}
    =
    (c_1A_{k1}+c_2A_{k2})
    \begin{pmatrix}
        \sqrt{D} &-\sqrt{D} \\
        1 & 1
    \end{pmatrix},
\label{eq: B_k cA}\end{equation}
where
\begin{equation}
    A_{ki}=
    \begin{pmatrix}
        b_{ki1} & a_{ki1} \\
        b_{ki2} & a_{ki2}
    \end{pmatrix}.
\end{equation}
Hence,
\begin{equation}
    B_k\begin{pmatrix}
        c_1 \\
        c_2
    \end{pmatrix}
    =
    \gamma^{-1}\begin{pmatrix}
        1 \\
        0
    \end{pmatrix},
\label{eq: B_k c_1 c_2}\end{equation}
where 
\begin{equation}
  \label{eq:Bkdef}
    B_k=
    \begin{pmatrix}
        b_{k11} & b_{k21} \\
        b_{k12} & b_{k22}
    \end{pmatrix}.
  \end{equation}
That $c_1$ and $c_2$ are in fact integers now follows immediately from Lemma \ref{lm: B_k in GL(2,Z)}, which we prove later.
\end{proof}

\begin{lm}
    Let $M\in \mathrm{M}_2(\mathbb{Z})$. If the entries of $M\begin{pmatrix}
        a \\
        b
    \end{pmatrix}$ are coprime for any coprime integers $a$ and $b$, then $M\in\GL(2,\mathbb{Z})$.
    \label{lm: M in GL(2,Z)}\end{lm}

\begin{proof}
  As $M$ is a $2\times 2$ matrix with integer entries, it can be written in the Smith normal form 
\begin{equation}
    M=S_1\begin{pmatrix}
            \lambda_1 & 0 \\
            0 & \lambda_2
        \end{pmatrix}S_2,
\end{equation}
where $\lambda_1, \lambda_2$ are non-negative integers and $S_1, S_2\in\GL(2,\mathbb{Z})$.
By considering the first column of the diagonal matrix, we have 
\begin{equation}
    MS_2^{-1}\begin{pmatrix}
        1 \\
        0
    \end{pmatrix}
    =
    M\begin{pmatrix}
        a \\
        b
    \end{pmatrix}
    =\lambda_1S_1\begin{pmatrix}
        1 \\
        0
    \end{pmatrix},
\end{equation}
for some coprime integers $a$ and $b$. With the assumption on $M$, we must have $\lambda_1=1$.
By considering the second column, we similarly obtain $\lambda_2=1$, and therefore, $M\in\GL(2,\mathbb{Z})$.
\end{proof}

\begin{lm}
    For any $k$, we have $B_k=
        \begin{pmatrix}
            b_{k11} & b_{k21} \\
            b_{k12} & b_{k22}
        \end{pmatrix}
        \in\GL(2,\mathbb{Z})$.
\label{lm: B_k in GL(2,Z)}\end{lm}

\begin{proof}
  Let $a, b$ be any coprime integers. Then $\xi'=a\beta_{k1}'+b\beta_{k2}'$ is a primitive vector in $I_k^{-1}$. Thus, the invertible ideal $\xi'I_k\subseteq\mathbb{Z}[\sqrt{D}]$ has no rational integer divisors greater than $1$. By Proposition \ref{prp:general m not equiv 2 mod 4}, $\xi'I_k$ has a basis of the form 
\begin{equation}
  \label{eq:xiprimeIk}
    \begin{pmatrix}
        1 & \mu' \\
        0 & m'
    \end{pmatrix}
    \begin{pmatrix}
        \sqrt{D} & -\sqrt{D} \\
        1 & 1
      \end{pmatrix}
      = \gamma \mathfrak{B}_k
      \begin{pmatrix}
        \xi' & 0 \\
        0 & \overline{\xi'}
      \end{pmatrix}
      ,
    \end{equation}
arguing as we did in the proof of the first part of Proposition \ref{prp: B_k}.
By considering the first column of \eqref{eq:xiprimeIk}, we have
\begin{equation}
  \label{eq:firstcolumn}
  B_k
  \begin{pmatrix}
        a \\
        b
    \end{pmatrix}
    =
    \gamma^{-1}\begin{pmatrix}
        1 \\
        0
    \end{pmatrix},
\end{equation}
for some $\gamma\in\GL(2,\mathbb{Z})$.
As $a$ and $b$ are arbitrary coprime integers, \eqref{eq:firstcolumn} shows that $B_k$ satisfies the condition in Lemma \ref{lm: M in GL(2,Z)}.
\end{proof}

We now turn to ideals in $\mathbb{Z}[\tfrac{1 + \sqrt{D}}{2}]$.
For $1\leq l \leq h_2^+(D)$, we fix representatives $J_l$ of the narrow class group of $\mathbb{Z}[\frac{1+\sqrt{D}}{2}]$.
As before, for any ideal $J\subseteq\mathbb{Z}[\frac{1+\sqrt{D}}{2}]$ there exists a totally positive $\zeta\in J_l^{-1}$ such that $J=\zeta J_l$.
Using the embedding $\mathbb{Q}(\sqrt{D}) \to \mathbb{R}^2$,
we associate a $\mathbb{Z}$-basis $\{\delta_{l1},\delta_{l2}\}$ of $J_l$ with the matrix
\begin{equation}
    \begin{pmatrix}
        \delta_{l1} \\
        \delta_{l2}
    \end{pmatrix}
    \rightarrow \mathfrak{D}_l=
    \begin{pmatrix}
        \delta_{l1} & \overline{\delta}_{l1}\\
        \delta_{l2} & \overline{\delta}_{l2}
      \end{pmatrix}
      ,
    \end{equation}
and as before we may assume $\det(\mathfrak{D}_l)>0$. 
Thus, the basis $\{\zeta\delta_{l1},\zeta\delta_{l2}\}$ of $\zeta J_l$ has basis matrix
\begin{equation}
    \begin{pmatrix}
        \zeta\delta_{l1} & \overline{\zeta}\overline{\delta}_{l1}\\
        \zeta\delta_{l2} & \overline{\zeta}\overline{\delta}_{l2}
    \end{pmatrix}
    =\mathfrak{D}_l
    \begin{pmatrix}
        \zeta & 0 \\
        0 & \overline{\zeta}
    \end{pmatrix}.
  \end{equation}
  
Letting $\varepsilon_2>1$ be the generator of the group of totally positive units of $\mathbb{Z}[\frac{1+\sqrt{D}}{2}]$. We define
\begin{equation}
\begin{aligned}
    \Gamma_{2;l}&=\left\{\pm\mathfrak{D}_l\begin{pmatrix}
        \varepsilon_2^s & 0\\
        0 & \varepsilon_2^{-s}
    \end{pmatrix}\mathfrak{D}_l^{-1}:s\in\mathbb{Z}\right\}\\
  &=\SL(2,\mathbb{Z})\cap\left\{\pm\mathfrak{D}_l
    \begin{pmatrix}
      t & 0 \\
      0 & t^{-1}
    \end{pmatrix}
    \mathfrak{D}_l^{-1}:t>0\right\},
\end{aligned}    
\end{equation}
since $J_l=\varepsilon J_l$ if and only if $\varepsilon$ is a unit of $\mathbb{Z}[\frac{1+\sqrt{D}}{2}]$.

Using Proposition \ref{prp:general m equiv 2 mod 4} in place of Proposition \ref{prp:general m not equiv 2 mod 4}, we obtain the following analogue of Proposition \ref{prp: B_k}.
We omit its proof as it is so similar to the proof above. 

\begin{prp}
  Let $D\equiv 1\pmod 4$ and suppose $\mu \bmod m$ satisfies $\mu^2\equiv D \pmod m$.
  If both $m$ and $\frac{D-\mu^2}{m}$ are even, then there exists a unique $1\leq l \leq h_2^+(D)$ and a unique double coset $\Gamma_\infty\gamma\Gamma_{2;l}\in\Gamma_\infty\backslash\SL(2,\mathbb{Z})\slash\Gamma_{2;l}$ such that 
    \begin{equation}
        \begin{pmatrix}
            1 & \frac{\mu-1}{m}\\
            0 & 1
        \end{pmatrix}
        \begin{pmatrix}
            1 & 0 \\
            0 & \frac{m}{2}
        \end{pmatrix}
        \begin{pmatrix}
        \frac{1+\sqrt{D}}{2} & \frac{1-\sqrt{D}}{2}\\
        1 & 1
        \end{pmatrix}
        =\gamma\mathfrak{D}_l
        \begin{pmatrix}
            \zeta & 0 \\
            0 & \overline{\zeta}
        \end{pmatrix}
    \label{eq: gamma D_l}\end{equation}
    for some totally positive $\zeta$.
    
    Conversely, if for $\Gamma_\infty\gamma\Gamma_{2;l}\in\Gamma_\infty\backslash\SL(2,\mathbb{Z})\slash\Gamma_{2;l}$ there exists positive real numbers $\zeta_1$ and $\zeta_2$ such that
    \begin{equation}
        \gamma\mathfrak{D}_l
        \begin{pmatrix}
        \zeta_1 & 0 \\
        0 & \zeta_2
        \end{pmatrix}
        =
        \begin{pmatrix}
        1 & * \\
        0 & * \\
        \end{pmatrix}
        \begin{pmatrix}
            \frac{1+\sqrt{D}}{2} & \frac{1-\sqrt{D}}{2}\\
            1 & 1
        \end{pmatrix},
    \end{equation}
    then in fact $\zeta_1 = \zeta$ and $\zeta_2 = \overline{\zeta}$ for some totally positive $\zeta \in J_l^{-1}$ and we have
    \begin{equation}
        \gamma\mathfrak{D}_l
        \begin{pmatrix}
        \zeta & 0 \\
        0 & \overline{\zeta}
        \end{pmatrix}
        =
        \begin{pmatrix}
        1 & \frac{\mu-1}{m}\\
        0 & 1
        \end{pmatrix}
        \begin{pmatrix}
            1 & 0 \\
            0 & \frac{m}{2}
        \end{pmatrix}
        \begin{pmatrix}
            \frac{1+\sqrt{D}}{2} & \frac{1-\sqrt{D}}{2}\\
            1 & 1
        \end{pmatrix},
    \end{equation}
    where $\mu \bmod m$ satisfies $\mu^2\equiv D\pmod m$ and both $m$ and $\frac{D-\mu^2}{m}$ are even.
\label{prp: D_l}\end{prp}

\subsection{Roots and geodesics}

We define oriented geodesics $\bm{c}_{I_k}$ for $1\leq k \leq h_1^+(D)$, and $\bm{c}_{J_l}$ for $1\leq l \leq h_2^+(D)$ in $\mathbb{H}$ by
\begin{equation}
    \bm{c}_{I_k}=\left\{(\det\mathfrak{B}_k)^{-\frac{1}{2}}\mathfrak{B}_k
      \begin{pmatrix}
        t & 0 \\
        0 & t^{-1}
      \end{pmatrix}
\mathrm{i} : t>0\right\}
\end{equation}
and
\begin{equation}
  \bm{c}_{J_l}=\left\{(\det\mathfrak{D}_l)^{-\frac{1}{2}}\mathfrak{D}_l
    \begin{pmatrix}
      t & 0 \\
      0 & t^{-1}
    \end{pmatrix}
    \mathrm{i} : t>0\right\}.
\end{equation}
Recalling that the top $z_{\bm{c}} \in \mathbb{H}$ of a positively oriented geodesic $\bm{c}$ is the point on $\bm{c}$ which has the largest imaginary part on $\bm{c}$, we now interpret Propositions \ref{prp: B_k} and \ref{prp: D_l} geometrically as follows.

\begin{prp}
     Let $D\equiv 1\pmod 4$ and let $\mu \bmod m$ satisfy $\mu^2\equiv D \pmod m$. If either $m$ or $\frac{D-\mu^2}{m}$ is odd, then there exists a unique $1\leq k\leq h_1^+(D)$ and a unique double coset $\Gamma_\infty\gamma\Gamma_{1;k}\in\Gamma_\infty\backslash \Gamma \slash\Gamma_{1;k}$ such that
     \begin{equation}
       \label{eq:rootstops1}
         z_{\gamma\bm{c}_{I_k}}\equiv\frac{\mu}{m}+\mathrm{i}\frac{\sqrt{D}}{m} \pmod {\Gamma_\infty}.
     \end{equation}
     Conversely, given $1\leq k\leq h_1^+(D)$ and a double coset $\Gamma_\infty\gamma\Gamma_{1;k}$ such that $\gamma\bm{c}_{I_k}$ is positively oriented, there exists a unique positive integer $m$ and a residue class $\mu \bmod m$ with either $m$ or $\frac{D-\mu^2}{m}$ odd satisfying \eqref{eq:rootstops1}.
\label{prp: congruence I SL(2,Z)}\end{prp}

\begin{proof}
  In the identification of $G/\{\pm I\}$ with the unit tangent bundle $\mathrm{T}_1(\mathbb{H})$, the matrices
\begin{equation}
    \begin{pmatrix}
        * & * \\
        0 & *
    \end{pmatrix}
    \begin{pmatrix}
        \cos\frac{\pi}{4} & -\sin\frac{\pi}{4} \\
        \sin\frac{\pi}{4} & \cos\frac{\pi}{4}
    \end{pmatrix}
\label{eq: Iwasawa matrices}\end{equation}
are identified with points in $\mathbb{H}$ whose tangent vectors point horizontally to the right.
The intersection, if it exists, of these matrices with those of the form
\begin{equation}
    (\det\mathfrak{B}_k)^{-\frac{1}{2}}\gamma\mathfrak{B}_k\begin{pmatrix}
        * & 0 \\
        0 & *
    \end{pmatrix}
\end{equation}
is the point on $\gamma\bm{c}_{I_k}$ which has a right horizontal tangent vector. Such a point exists exactly when $\gamma\bm{c}_{I_k}$ is positively oriented, and its imaginary part is the largest on the geodesic $\gamma\bm{c}_{I_k}$. Scaling and rewriting \eqref{eq: gamma B_k}, we have
\begin{equation*}
  (\det\mathfrak{B}_k)^{-\frac{1}{2}}\gamma\mathfrak{B}_k
  \begin{pmatrix}
    t & 0 \\
    0 & t^{-1}
  \end{pmatrix}
    =
    \begin{pmatrix}
         1 & \frac{\mu}{m} \\
         0 & 1 \\
    \end{pmatrix}
    \begin{pmatrix}
         \left(\frac{\sqrt{D}}{m}\right)^{\frac{1}{2}} & 0 \\
         0 & \left(\frac{\sqrt{D}}{m}\right)^{-\frac{1}{2}} \\
    \end{pmatrix}
    \begin{pmatrix}
        \frac{1}{\sqrt{2}} & -\frac{1}{\sqrt{2}} \\
        \frac{1}{\sqrt{2}} & \frac{1}{\sqrt{2}} \\
    \end{pmatrix},
\end{equation*}
finishing the proof via Proposition $\ref{prp: B_k}$.
\end{proof}

We also record the following proposition, which has a nearly identical proof, the only difference being the rewriting
\begin{multline}
    \begin{pmatrix}
        1 & \frac{\mu-1}{2} \\
        0 & \frac{m}{2}
    \end{pmatrix}
    \begin{pmatrix}
        \frac{1+\sqrt{D}}{2} & \frac{1-\sqrt{D}}{2} \\
        1 & 1
    \end{pmatrix}
    \\
    =  ( \tfrac{1}{2} m \sqrt{D})^{\frac{1}{2}}
    \begin{pmatrix}
        1 & \frac{\mu}{m} \\
        0 & 1
    \end{pmatrix}
    \begin{pmatrix}
      \left(\frac{\sqrt{D}}{m}\right)^{\frac{1}{2}} & 0 \\
      0 & \left(\frac{\sqrt{D}}{m}\right)^{-\frac{1}{2}}
    \end{pmatrix}
    \begin{pmatrix}
        \frac{1}{\sqrt 2} & -\frac{1}{\sqrt 2} \\
        \frac{1}{\sqrt 2} & \frac{1}{\sqrt 2}
    \end{pmatrix}.
\end{multline}

\begin{prp}
  Let $D\equiv 1\pmod 4$ and let $\mu \bmod m$ satisfy $\mu^2\equiv D \pmod m$ with both $m$ and $\frac{D-\mu^2}{m}$ are even.
  Then there exists a unique $1\leq l\leq h_2^+(D)$ and a unique double coset $\Gamma_\infty\gamma\Gamma_{2;l}\in\Gamma_\infty\backslash \Gamma \slash\Gamma_{2;l}$ such that
  \begin{equation}
    \label{eq:rootstops2}
         z_{\gamma\bm{c}_{J_l}}\equiv\frac{\mu}{m}+\mathrm{i}\frac{\sqrt{D}}{m} \pmod {\Gamma_\infty}.
     \end{equation}
     Conversely, given $1\leq l\leq h_2^+(D)$ and a double coset $\Gamma_\infty\gamma\Gamma_{2;l}$ such that $\gamma\bm{c}_{J_l}$ is positively oriented, there exists a unique positive integer $m$ and a residue class $\mu \bmod m$ with both $m$ and $\frac{D-\mu^2}{m}$ even satisfying \eqref{eq:rootstops2}.
\label{prp: congruence J SL(2,Z)}\end{prp}

\section{Congruence subgroups}
\label{sec: Congruence subgroups}

In the previous section, we considered the $\Gamma = \SL(2,\mathbb{Z})$-orbits of geodesics.
In this section, we work with the orbits under the action of the congruence subgroups
\begin{equation}
    \Gamma_0(n)=\left\{\begin{pmatrix}
        a & b \\
        c & d
    \end{pmatrix}\in\SL(2,\mathbb{Z}):c\equiv 0 \pmod n\right\}.
\end{equation}
We find below that given a fixed root $\nu \bmod n$, we can generate all roots $\mu \bmod m$ satisfying $m\equiv 0\pmod n$ and $\mu\equiv\nu\pmod n$ as follows:
To obtain those roots for which either $m$ or $\frac{D-\mu^2}{m}$ is odd, we pick cosets $\Gamma_0(n)\gamma_{1;k}$ (depending on $n$ and $\nu$) in $\Gamma_0(n)\backslash\Gamma$ and consider the $\Gamma_0(n)$-orbits of the geodesics $\gamma_{1;k}\bm{c}_{I_k}$, and similarly, for the remaining roots for which both $m$ and $\frac{D-\mu^2}{m}$ are even, we choose coset representatives $\gamma_{2;l}$ (again depending on $n$ and $\nu$) for either $\Gamma_0(n)\backslash\Gamma$ or $\Gamma_0(\frac{n}{2})\backslash\Gamma$ (depending on the parity of $n$) and correspondingly consider the $\Gamma_0(n)$-orbits or $\Gamma_0(\frac{n}{2})$-orbits of the geodesics $\gamma_{2;l}\bm{c}_{J_l}$.

We recall that $\{\beta_{k1}, \beta_{k2}\}$ and $\{\beta'_{k1}, \beta_{k2}'\}$ are respectively bases of the invertible ideals $I_k$ and $I_k^{-1}$ of $\mathbb{Z}[\sqrt{D}]$.
We further recall the integer coefficients $a_{kij}, b_{kij}$ determined by  $\beta'_{ki}\beta_{kj}=a_{kij}+b_{kij}\sqrt{D}$.
Similarly, for $\mathbb{Z}[\tfrac{1+\sqrt{D}}{2}]$, we have bases $\{\delta_{l1}, \delta_{l2}\}$ and $\{\delta'_{l1}, \delta_{l2}'\}$ of the ideals $J_l$ and $J_l^{-1}$ respectively.
Here, the analogous integer coefficients $c_{lij}, d_{lij}$ are determined by $\delta'_{li}\delta_{lj}=c_{lij}+d_{lij}\frac{1+\sqrt{D}}{2}$.

Finally, we recall the matrix $B_k$ defined in \eqref{eq:Bkdef} and proven to be in $\GL(2, \mathbb{Z})$ in Lemma \ref{lm: B_k in GL(2,Z)}.
We similarly define
\begin{equation}
  \label{eq:Dldef}
  D_l =
  \begin{pmatrix}
    d_{l11} & d_{l21} \\
    d_{l12} & d_{l22}
  \end{pmatrix}
  ,
\end{equation}
which can be seen to be in $\GL(2, \mathbb{Z})$ by the same reasoning as in Lemma \ref{lm: B_k in GL(2,Z)}.

\begin{cor}
  As $\gamma=\begin{pmatrix} * & * \\ c & d \end{pmatrix}$ runs through the representatives of the double coset $\Gamma_\infty\backslash\Gamma \slash\Gamma_{1;k}$ such that $\gamma\bm{c}_{I_k}$ is positively oriented, 
  \begin{equation}
        \begin{pmatrix}
            \mu \\
            m
        \end{pmatrix}
        =
        \det(B_k)\gamma
        \begin{pmatrix}
            a_{k11}b_{k21}-a_{k21}b_{k11} & a_{k11}b_{k22}-a_{k21}b_{k12} \\
            a_{k12}b_{k21}-a_{k22}b_{k11} & a_{k12}b_{k22}-a_{k22}b_{k12}
        \end{pmatrix}
        \begin{pmatrix}
            c \\
            d
        \end{pmatrix}
    \label{eq: gamma c_{I_k}}\end{equation}
  parametrizes all $\mu$ and $m$ satisfying $\mu^2\equiv D\pmod m$ with either $m$ or $\frac{D-\mu^2}{m}$ odd.

  Similarly, as $\gamma=\begin{pmatrix} * & * \\ c & d \end{pmatrix}$ runs through the representatives of the double coset ${\Gamma_\infty\backslash\Gamma \slash\Gamma_{2;l}}$ such that $\gamma\bm{c}_{J_l}$ is positively oriented, 
  \begin{equation}
        \begin{pmatrix}
            \frac{\mu-1}{2} \\
            \frac{m}{2}
        \end{pmatrix}
        =
        \det(D_l)\gamma
        \begin{pmatrix}
            c_{l11}d_{l21}-c_{l21}d_{l11} & c_{l11}d_{l22}-c_{l21}d_{l12} \\
            c_{l12}d_{l21}-c_{l22}d_{l11} & c_{l12}d_{l22}-c_{l22}d_{l12}
        \end{pmatrix}
        \begin{pmatrix}
            c \\
            d
        \end{pmatrix}
    \label{eq: gamma c_{J_l}}\end{equation}
    parametrizes all $\mu$ and $m$ satisfying $\mu^2\equiv D\pmod m$ with both $m$ and $\frac{D-\mu^2}{m}$ even.
    \label{cor: gamma c_{I_k}}
\end{cor}
  
\begin{proof}
  Starting with \eqref{eq: gamma c_{I_k}}, let $\xi=c_1\beta_{k1}'+c_2\beta_{k2}'$.
  By Proposition \ref{prp: B_k} and \eqref{eq: B_k cA}, we have 
\begin{equation}
    \begin{aligned}
        \gamma \mathfrak{B}_k
        \begin{pmatrix}
            \xi & 0 \\
            0 & \overline{\xi}
        \end{pmatrix}
        &=\gamma (c_1A_{k1}+c_2A_{k2})
        \begin{pmatrix}
            \sqrt{D} & -\sqrt{D}\\
            1 & 1
        \end{pmatrix}\\
        &=
        \begin{pmatrix}
            1 & \mu \\
            0 & m \\
        \end{pmatrix}
        \begin{pmatrix}
            \sqrt{D} & -\sqrt{D}\\
            1 & 1
        \end{pmatrix},
    \end{aligned}    
\end{equation}
where $\gamma=\begin{pmatrix} * & * \\ c & d \end{pmatrix}\in\SL(2,\mathbb{Z})$. This yields 
\begin{equation}
    \gamma(c_1A_{k1}+c_2A_{k2})=\gamma
    \begin{pmatrix}
        c_1b_{k11}+c_2b_{k21} & c_1a_{k11}+c_2a_{k21} \\
        c_1b_{k12}+c_2b_{k22} & c_1a_{k12}+c_2a_{k22}
    \end{pmatrix}
    =
    \begin{pmatrix}
        1 & \mu \\
        0 & m \\
    \end{pmatrix}.
\label{eq: 1 0 mu m}\end{equation}
Considering the first column of \eqref{eq: 1 0 mu m}, we obtain
\begin{equation}
    \gamma B_k
    \begin{pmatrix}
        c_1\\
        c_2
    \end{pmatrix}
    =
    \gamma
    \begin{pmatrix}
        c_1b_{k11}+c_2b_{k21} \\
        c_1b_{k12}+c_2b_{k22}
    \end{pmatrix}
    =
    \begin{pmatrix}
        1 \\
        0
    \end{pmatrix},
\end{equation}
which implies
\begin{equation}
    \begin{aligned}
        \begin{pmatrix}
            c_1 \\
            c_2
        \end{pmatrix}
        =B_k^{-1}\gamma^{-1}
        \begin{pmatrix}
            1 \\
            0
        \end{pmatrix}
        &=\frac{1}{\det(B_k)}
        \begin{pmatrix}
            b_{k22} & -b_{k21} \\
            -b_{k12} & b_{k11}
        \end{pmatrix}
        \begin{pmatrix}
            d \\
            -c
        \end{pmatrix}\\
        &=\det(B_k)
        \begin{pmatrix}
            b_{k21} & b_{k22} \\
            -b_{k11} & -b_{k12}
        \end{pmatrix}
        \begin{pmatrix}
            c \\
            d
        \end{pmatrix}
    \end{aligned}    
\end{equation}
as $B_k\in\GL(2,\mathbb{Z})$ and $\gamma^{-1}=\begin{pmatrix} d & * \\ -c & * \end{pmatrix}$. Equating the second column of \eqref{eq: 1 0 mu m}, we conclude
\begin{equation}
    \begin{pmatrix}
        \mu \\
        m
    \end{pmatrix}
    =\gamma
    \begin{pmatrix}
        a_{k11} & a_{k21} \\
        a_{k12} & a_{k22}
    \end{pmatrix}
    \begin{pmatrix}
        c_1 \\
        c_2
    \end{pmatrix}
    =\det(B_k)\gamma
    \begin{pmatrix}
        a_{k11} & a_{k21} \\
        a_{k12} & a_{k22}
    \end{pmatrix}
    \begin{pmatrix}
        b_{k21} & b_{k22} \\
        -b_{k11} & -b_{k12}
    \end{pmatrix}
    \begin{pmatrix}
        c \\
        d
    \end{pmatrix}
\end{equation}
so the corollary holds.

The parametrization \eqref{eq: gamma c_{J_l}} follows by the same reasoning.
\end{proof}

Corollary \ref{cor: gamma c_{I_k}} leads to 
Lemma \ref{lm: invariant} below, which says that the quadratic conditions $m \equiv 0 \pmod n$ and $\mu \equiv \nu \pmod n$ are invariant either under $\Gamma_0(n)$ or $\Gamma_0(\frac{n}{2})$.
By invariant, we mean that if $\mu \bmod m$ satisfies the congruence conditions and corresponds to some $\gamma$ via Propositions \ref{prp: congruence I SL(2,Z)} or \ref{prp: congruence J SL(2,Z)}, then any root $\mu' \bmod {m'}$ corresponding to some $\gamma' \in \Gamma_0(n) \gamma$ (or $\Gamma( \frac{n}{2}) \gamma$, see below) also satisfies the congruence conditions $m' \equiv 0 \pmod n$ and $\mu' \equiv \nu \pmod n$.


\begin{lm}
  Let $D\equiv 1\pmod 4$ and let $\mu \bmod m$, $\nu \bmod n$ satisfy $\mu^2\equiv D \pmod m$, $\nu^2\equiv D\pmod n$.
    \begin{enumerate}[(i)]
        \item Suppose that either $m$ or $\frac{D-\mu^2}{n}$ is odd. Then, the conditions $m\equiv 0\pmod n$ and $\mu\equiv\nu \pmod n$ are invariant under $\Gamma_0(n)$.
        \item Suppose that both $m$ and $\frac{D-\mu^2}{n}$ are even. If either $n$ or $\frac{D-\nu^2}{n}$ is odd, then the conditions $m\equiv 0\pmod n$ and $\mu\equiv\nu \pmod n$ are invariant under $\Gamma_0(n)$. On other hand, if both $n$ and $\frac{D-\nu^2}{n}$ are even, then the quadratic conditions $m\equiv 0\pmod n$ and $\mu\equiv\nu \pmod n$ are invariant under $\Gamma_0(\frac{n}{2})$.
    \end{enumerate}
\label{lm: invariant}\end{lm}

\begin{proof}
  (i) First, we assume that either $m$ or $\frac{D-\mu^2}{2}$ is odd. By Corollary \ref{cor: gamma c_{I_k}}, we can write
\begin{equation}
    \begin{pmatrix}
        \mu \\
        m
    \end{pmatrix}
    =
    \begin{pmatrix}
        a & b \\
        c & d 
    \end{pmatrix}
    \begin{pmatrix}
        A & B \\
        C & D
    \end{pmatrix}
    \begin{pmatrix}
        c \\
        d
    \end{pmatrix}
    \quad\text{for some}\quad \gamma=\begin{pmatrix}
        a & b \\
        c & d 
    \end{pmatrix}\in\SL(2,\mathbb{Z}).
\end{equation}
Given $\gamma_0=\begin{pmatrix}
        a' & b' \\
        nc' & d' 
    \end{pmatrix}\in\Gamma_0(n)$, we have $\gamma_0\gamma=\begin{pmatrix}
        a'a+b'c & a'b+b'd \\
        nc'a+d'c & nc'b+d'd 
    \end{pmatrix}$ and 
\begin{equation}
    \begin{pmatrix}
        \mu' \\
        m'
    \end{pmatrix}
    =
    \begin{pmatrix}
        a'a+b'c & a'b+b'd \\
        nc'a+d'c & nc'b+d'd 
    \end{pmatrix}
    \begin{pmatrix}
        A & B \\
        C & D
    \end{pmatrix}
    \begin{pmatrix}
        nc'a+d'c \\
        nc'b+d'd 
    \end{pmatrix}
\end{equation}
so $m' \equiv d'^2m \equiv 0 \pmod n$ and $\mu'\equiv (1+nc'b')\mu+b'd'm \equiv \mu \pmod n$.

\bigskip

(iia) Now, we assume that both $m$ and $\frac{D-\mu^2}{2}$ are even. By Corollary \ref{cor: gamma c_{I_k}}, we can write
\begin{equation}
    \begin{pmatrix}
        \frac{\mu-1}{2} \\
        \frac{m}{2}
    \end{pmatrix}
    =
    \begin{pmatrix}
        a & b \\
        c & d 
    \end{pmatrix}
    \begin{pmatrix}
        A & B \\
        C & D
    \end{pmatrix}
    \begin{pmatrix}
        c \\
        d
    \end{pmatrix}
    \quad\text{for some}\quad \gamma=\begin{pmatrix}
        a & b \\
        c & d 
    \end{pmatrix}\in\SL(2,\mathbb{Z}).
\end{equation}
Suppose that either $n$ or $\frac{D-\nu^2}{n}$ is odd, and we note that if also $m\equiv 0 \pmod n$ and $\mu \equiv \nu \pmod n$, then in fact $n$ is odd since $\nu^2 - D \equiv \mu^2 - D \equiv 0 \pmod {2n}$.
Given $\gamma_0=\begin{pmatrix}
        a' & b' \\
        nc' & d' 
    \end{pmatrix}\in\Gamma_0(n)$, we have
\begin{equation}
    \begin{pmatrix}
        \frac{\mu'-1}{2} \\
        \frac{m'}{2}
    \end{pmatrix}
    =
    \begin{pmatrix}
        a'a+b'c & a'b+b'd \\
        nc'a+d'c & nc'b+d'd 
    \end{pmatrix}
    \begin{pmatrix}
        A & B \\
        C & D
    \end{pmatrix}
    \begin{pmatrix}
        nc'a+d'c \\
        nc'b+d'd 
    \end{pmatrix}
\end{equation}
so $\frac{m'}{2} \equiv d'^2\frac{m}{2} \equiv 0 \pmod n$ and $\frac{\mu'-1}{2}\equiv (1+nc'b')\frac{\mu-1}{2}+b'd'\frac{m}{2} \equiv \frac{\mu-1}{2} \pmod n$ since $m\equiv 0\pmod {2n}$.
This implies $m'\equiv m\equiv 0\pmod n$ and $\mu'\equiv\mu\equiv\nu\pmod n$.

\bigskip

(iib) Suppose that both $n$ and $\frac{D-\nu^2}{n}$ are even. Given $\gamma_0=\begin{pmatrix}
        a' & b' \\
        \frac{n}{2}c' & d' 
    \end{pmatrix}\in\Gamma_0(\frac{n}{2})$, we have
\begin{equation}
    \begin{pmatrix}
        \frac{\mu'-1}{2} \\
        \frac{m'}{2}
    \end{pmatrix}
    =
    \begin{pmatrix}
        a'a+b'c & a'b+b'd \\
        \frac{n}{2}c'a+d'c & \frac{n}{2}c'b+d'd 
    \end{pmatrix}
    \begin{pmatrix}
        A & B \\
        C & D
    \end{pmatrix}
    \begin{pmatrix}
        \frac{n}{2}c'a+d'c \\
        \frac{n}{2}c'b+d'd 
    \end{pmatrix}.
\end{equation}    
Multiplying both sides by the scalar $2$, the equality becomes 
\begin{equation}
    \begin{pmatrix}
        \mu'-1 \\
        m'
    \end{pmatrix}
    =
    \begin{pmatrix}
        a'a+b'c & a'b+b'd \\
        \frac{n}{2}c'a+d'c & \frac{n}{2}c'b+d'd 
    \end{pmatrix}
    \begin{pmatrix}
        A & B \\
        C & D
    \end{pmatrix}
    \begin{pmatrix}
        nc'a+2d'c \\
        nc'b+2d'd 
    \end{pmatrix}
\end{equation}
so $m' \equiv d'^2m \equiv 0\pmod n$ and $\mu'\equiv \mu+nc'b'\frac{\mu-1}{2}+b'd'm\equiv\mu \pmod n$.
\end{proof}

We now prove the existence of elements $\gamma_{1;k}$ and $\gamma_{2;l}$ (depending on $n$ and $\nu$) of $\Gamma$ referred to at the beginning of this section.
Specifically, these $\gamma_{1;k}$ and $\gamma_{2;l}$ should be so that the tops $z_{\gamma_{1;k} \bm{c}_{1;k}}$ and $z_{\gamma_{2;l} \bm{c}_{2;l}}$ have the form $\frac{\mu}{m} + \mathrm{i}\frac{\sqrt{D}}{m}$ with $m \equiv 0 \pmod n$ and $\mu \equiv \nu \pmod n$.
Lemma \ref{lm: invariant} then implies that the top of any positively oriented geodesic in the same $\Gamma_0(n)$-orbit or $\Gamma_0(\frac{n}{2})$-orbit as $\gamma_{1;k}$ or $\gamma_{2;l}$ also satisfies these congruence conditions.
We note that it is enough to show the existence of ideal classes having the corresponding properties. 

\begin{lm}
  Let $D\equiv 1\pmod 4$ and $\nu \bmod n$ satisfy $\nu^2\equiv D\pmod n$. For any $1\leq k \leq h_1^+(D)$, there exists $I_{k,\nu,n}$, equivalent to $I_k$ in the narrow class group of $\mathbb{Z}[\sqrt{D}]$, such that $I_{k,\nu,n}$ has a basis given by $\begin{pmatrix} 1 & \mu \\ 0 & m \end{pmatrix}\begin{pmatrix} \sqrt{D} \\ 1 \end{pmatrix}$ where $m\equiv 0 \pmod n$ and $\mu\equiv\nu \pmod n$ with either $m$ or $\frac{D-\mu^2}{m}$ odd.

  Similarly, for any $1\leq l \leq h_2^+(D)$, there exists $J_{l,\nu,n}$, equivalent to $J_l$ in the narrow class group of $\mathbb{Z}[\frac{1+\sqrt{D}}{2}]$, such that $J_{l,\nu,n}$ has a basis given by $\begin{pmatrix} 1 & \frac{\mu-1}{2} \\ 0 & \frac{m}{2} \end{pmatrix}\begin{pmatrix} \frac{1+\sqrt{D}}{2} \\ 1 \end{pmatrix}$ where $m\equiv 0 \pmod n$ and $\mu\equiv\nu \pmod n$ with both $m$ and $\frac{D-\mu^2}{m}$ even.
\label{lm: I_k,nu,n}\end{lm}

\begin{proof}
  We begin with the ideal $I_k$ of $\mathbb{Z}[\sqrt D]$, and in addition, suppose that either $n$ or $\frac{D-\nu^2}{n}$ is odd.
  By Proposition \ref{prp:general m not equiv 2 mod 4}, $I = (n,\nu+\sqrt{D})$ is an invertible ideal of $\mathbb{Z}[\sqrt{D}]$.
  Let $I_0=(m_0,\mu_0+\sqrt{D})$ with $\mu_0^2\equiv D\pmod {m_0}$ be an invertible ideal of $\mathbb{Z}[\sqrt{D}]$, equivalent to $I_kI^{-1}$ in the narrow class group of $\mathbb{Z}[\sqrt{D}]$, such that $I_0$ is not divisible by rational integers greater than $1$ and $\gcd(m_0,n)=1$ (the existence of such an ideal $I_0$ can be seen, for example, by observing that the possible norms $m_0$ in the class of $I_k I^{-1}$ are given by the values of a primitive binary quadratic form at coprime integers, see Corollary \ref{cor: gamma c_{I_k}}).
  As $m_0$ and $n$ are coprime, there exist integers $\overline{m}_0$ and $\overline{n}$ such that $m_0\overline{m}_0+n\overline{n}=1$.
  We have
\begin{equation}
\begin{aligned}
    I_0I&=(m_0n,m_0\nu+m_0\sqrt{D},n\mu_0+n\sqrt{D},\mu_0\nu+D+(\mu_0+\nu)\sqrt{D}) \\
    &=(m_0n,m_0n(\nu-\mu_0),m_0\overline{m}_0\nu+n\overline{n}\mu_0+\sqrt{D},\mu_0\nu+D+(\mu_0+\nu)\sqrt{D})\\
    &=(m_0n,m_0\overline{m}_0\nu+n\overline{n}\mu_0+\sqrt{D},\mu_0\nu+D-(\mu_0+\nu)(m_0\overline{m}_0\nu+n\overline{n}\mu_0))\\
    &=(m_0n,m_0\overline{m}_0\nu+n\overline{n}\mu_0+\sqrt{D}),
\end{aligned}
\label{eq: I_0I_nu,n}\end{equation}
because $\mu_0\nu+D-(\mu_0+\nu)(m_0\overline{m}_0\nu+n\overline{n}\mu_0)\equiv 0 \pmod {m_0n}$.
Set $I_{k,\nu,n}=I_0I=(m,\mu+\sqrt{D})$, where $m=m_0n$ and $\mu\equiv m_0\overline{m}_0\nu+n\overline{n}\mu_0\pmod m$, i.e. $\mu \bmod m$ is obtained from $\nu \bmod n$ and $\mu_0 \bmod m_0 $ by the Chinese remainder theorem.

Suppose now that both $n$ and $\frac{D-\nu^2}{n}$ are even. Let $s\geq 1$ be the exponent such that $\frac{D-\nu^2}{2^sn}$ is odd.
By Proposition \ref{prp:general m not equiv 2 mod 4}, $I'=(n', \nu+\sqrt{D})$ is an invertible ideal of $\mathbb{Z}[\sqrt{D}]$, where $n'=2^sn$.
Let $I_0=(m_0,\mu_0+\sqrt{D})$ with $\mu_0^2\equiv D \pmod {m_0}$ be an invertible ideal of $\mathbb{Z}[\sqrt{D}]$, equivalent to $I_k(I')^{-1}$ in the narrow class group of $\mathbb{Z}[\sqrt{D}]$, such that $I_0$ is not divisible by rational integers greater than $1$ and $\gcd(m_0,n')=1$. As $m_0$ and $n'$ are coprime, there exist integers $\overline{m}_0$ and $\overline{n}'$ such that $m_0\overline{m}_0+n'\overline{n}'=1$.
Following the similar calculation in \eqref{eq: I_0I_nu,n}, we set $I_{k,\nu,n}=I_0I'=(m,\mu+\sqrt{D})$, where $m=2^sm_0n$ and $\mu\equiv m_0\overline{m}_0\nu+n'\overline{n}'\mu_0\pmod m$. 


We now turn to the ideal $J_l$ in $\mathbb{Z}[\tfrac{1 + \sqrt{D}}{2}]$, and we first suppose that both $n$ and $\frac{D-\nu^2}{n}$ are even.
By Proposition \ref{prp:general m equiv 2 mod 4}, $J=(\frac{n}{2},\frac{\nu-1}{2}+\frac{1+\sqrt{D}}{2})=\frac{1}{2}(n,\nu +\sqrt{D})$ is an ideal of $\mathbb{Z}[\frac{1+\sqrt{D}}{2}]$.
Let $J_0=\frac{1}{2}(m_0,\mu_0+\sqrt{D})$, where $\mu_0^2\equiv D\pmod {m_0}$ and $m_0\equiv 2\pmod 4$, be an ideal of $\mathbb{Z}[\frac{1+\sqrt{D}}{2}]$, equivalent to $J_lJ^{-1}$ in the narrow class group of $\mathbb{Z}[\frac{1+\sqrt{D}}{2}]$, such that $J_0$ is not divisible by rational integers greater than 1 and $\gcd(\frac{m_0}{2},\frac{n}{2})=1$ (the possible norms $\frac{m_0}{2}$ are given by values of a primitive binary quadratic form at coprime integers in this case as well).
As $\frac{m_0}{2}$ and $\frac{n}{2}$ are coprime, there exist integers $\overline{m}_0$ and $\overline{n}$ such that $\frac{1}{2}(m_0\overline{m}_0+n\overline{n})=1$. We have
\begin{equation}
\begin{aligned}
    J_0J&=\frac{1}{4}(m_0n,m_0\nu+m_0\sqrt{D},n\mu_0+n\sqrt{D},\mu_0\nu+D+(\mu_0+\nu)\sqrt{D}) \\
    &=\frac{1}{4}(m_0n,m_0n(\nu-\mu_0)/2,m_0\overline{m}_0\nu+n\overline{n}\mu_0+2\sqrt{D},\mu_0\nu+D+(\mu_0+\nu)\sqrt{D})\\
    &=\frac{1}{4}(m_0n,m_0\overline{m}_0\nu+n\overline{n}\mu_0+2\sqrt{D},\mu_0\nu+D+(\mu_0+\nu)\sqrt{D}) \\
    &=\frac{1}{2}(m_0n/2,(m_0\overline{m}_0\nu+n\overline{n}\mu_0)/2+\sqrt{D}),
\end{aligned}
\label{eq: J_0J_nu,n}\end{equation}
because $D + \mu_0\nu-(\mu_0+\nu)(m_0\overline{m}_0\nu+n\overline{n}\mu_0)/2 \equiv 0 \pmod {m_0n}$.
Set $J_{l,\nu,n}=J_0J=\frac{1}{2}(m,\mu+\sqrt{D})$, where $m=\frac{m_0n}{2}$ and $\mu\equiv (m_0\overline{m}_0\nu+n\overline{n}\mu_0)/2 \pmod m$.

Suppose now that either $n$ or $\frac{D-\nu^2}{n}$ is odd.
As before, since $m\equiv 0 \pmod n$, $\mu \equiv \nu \pmod n$ and both $m$ and $\frac{D-\mu^2}{m}$ are even, we have $n$ is odd.
Furthermore, we may assume that $\nu^2 \equiv D \pmod {2n}$ by possibly replacing $\nu$ with $\nu + n$. 
Then necessarily $\frac{D - \nu^2}{2n}$ is even, and by Proposition
\ref{prp:general m equiv 2 mod 4}, $J'=\frac{1}{2}(2n,\nu+\sqrt{D})$ is an ideal of $\mathbb{Z}[\frac{1+\sqrt{D}}{2}]$.
Let $J_0=\frac{1}{2}(m_0,\mu_0+\sqrt{D})$ with $\mu_0^2\equiv D\pmod {m_0}$ be an ideal of $\mathbb{Z}[\frac{1+\sqrt{D}}{2}]$, equivalent to $J_l(J')^{-1}$ in the narrow class group of $\mathbb{Z}[\frac{1+\sqrt{D}}{2}]$, such that $J_0$ is not divisible by rational integers greater than 1 and $\gcd(\frac{m_0}{2},n)=1$.
Following the similar calculation in \eqref{eq: J_0J_nu,n}, we set $J_{l,\nu,n}=J_0J'=\frac{1}{2}(m,\mu+\sqrt{D})$, where $m=m_0n$ and $\mu\equiv m_0\overline{m}_0\nu/2+n\overline{n}\mu_0 \pmod m$ with $m_0\overline{m}_0/2+n\overline{n}=1$.
\end{proof}





The next two lemmas show that if two tops $\frac{\mu_j}{m_j} + \mathrm{i} \frac{\sqrt{D}}{m_j}$ are in the same $\SL(2, \mathbb{Z})$-orbit and satisfy the congruence conditions $m_1 \equiv m_2 \equiv 0 \pmod n$ and $\mu_1 \equiv \mu_2 \equiv \nu \pmod n$, then in fact they are in the same $\Gamma_0(n)$-orbit or $\Gamma_0(\frac{n}{2})$-orbit. This completes the argument outlined at the beginning of this section.

\begin{lm}
    Let $D\equiv 1 \pmod 4$ and $m_1, m_2\in\mathbb{N}$. Let $\mu_1 \bmod {m_1}$ satisfy $\mu_1^2\equiv D \pmod {m_1}$ with either $m_1$ or $\frac{D-\mu_1^2}{m_1}$ odd, and let $\mu_2 \bmod {m_2}$ satisfy $\mu_2^2\equiv D \pmod {m_2}$ with either $m_2$ or $\frac{D-\mu_2^2}{m_2}$ odd. Suppose that there is a root $\nu \bmod n$ of $\nu^2\equiv D \pmod n$, and suppose that $m_1\equiv m_2 \equiv 0 \pmod n$ and $\mu_1\equiv \mu_2\equiv \nu \pmod n$. If there exists $\gamma\in\SL(2,\mathbb{Z})$ such that
    \begin{equation}
        \begin{pmatrix}
            1 & \mu_1 \\
            0 & m_1
        \end{pmatrix}
        \begin{pmatrix}
            \sqrt{D} & -\sqrt{D}\\
            1 & 1
        \end{pmatrix}
        \begin{pmatrix}
            \xi & 0 \\
            0 & \overline{\xi}
        \end{pmatrix}
        =\gamma
        \begin{pmatrix}
            1 & \mu_2 \\
            0 & m_2
        \end{pmatrix}
        \begin{pmatrix}
            \sqrt{D} & -\sqrt{D}\\
            1 & 1
        \end{pmatrix}  
    \label{eq: congruence subgroup I}\end{equation}
    for some totally positive $\xi$, then $\gamma\in\Gamma_0(n)$.
\label{lm: congruence subgroup I}\end{lm}

\begin{proof}
  By Proposition \ref{prp:general m not equiv 2 mod 4}, we may find two invertible ideals $I_1=(m_1, \mu_1+\sqrt{D})$ and $I_2=(m_2, \mu_2+\sqrt{D})$ of $\mathbb{Z}[\sqrt{D}]$ respectively corresponding to the roots $\mu_1 \bmod {m_1}$ and $\mu_2 \bmod {m_2}$.
  We note that \eqref{eq: congruence subgroup I} implies $\xi I_1=I_2$.

(i) If either $n$ or $\frac{D-\nu^2}{n}$ is odd, we  may have an invertible ideal $I$ of $\mathbb{Z}[\sqrt{D}]$ corresponding to the root $\nu \bmod n$ by Proposition \ref{prp:general m not equiv 2 mod 4} as well. Due to the congruence conditions modulo $n$, we may write
\begin{equation}
    \begin{pmatrix}
        1 & \mu_j \\
        0 & m_j
    \end{pmatrix}
    =
    \begin{pmatrix}
        1 & k_j \\
        0 & n_j
    \end{pmatrix}
    \begin{pmatrix}
        1 & \nu \\
        0 & n
    \end{pmatrix},
\end{equation}
where $n_j=\frac{m_j}{n}\in\mathbb{Z}$ and $k_j\in\mathbb{Z}$. This implies $I_j\subseteq I$ so $I_jI^{-1}\subseteq\mathbb{Z}[\sqrt{D}]$ for $j=1,2$. Let $\nu_1 \bmod n_1$ correspond to the invertible ideal $I_1I^{-1}$ of $\mathbb{Z}[\sqrt{D}]$. We have
\begin{equation}
    \begin{pmatrix}
        1 & \mu_1 \\
        0 & m_1
    \end{pmatrix}
    =
    \begin{pmatrix}
        1 & k \\
        0 & n
    \end{pmatrix}
    \begin{pmatrix}
        1 & \nu_1 \\
        0 & n_1
    \end{pmatrix}
\label{eq: nu_1 n_1}\end{equation}
for some $k\in\mathbb{Z}$. Since $\xi I_1I^{-1}=I_2I^{-1}\subseteq\mathbb{Z}[\sqrt{D}]$, then
\begin{equation}
    \begin{pmatrix}
        1 & \nu_1 \\
        0 & n_1
    \end{pmatrix}
    \begin{pmatrix}
        \sqrt{D} & -\sqrt{D} \\
        1 & 1
    \end{pmatrix}
    \begin{pmatrix}
        \xi & 0 \\
        0 & \overline{\xi}
    \end{pmatrix}
    =M
    \begin{pmatrix}
        \sqrt{D} & -\sqrt{D} \\
        1 & 1
    \end{pmatrix}
\label{eq: nu_1 n_1 M i}\end{equation}
for some integer matrix $M\in \mathrm{M}_2(\mathbb{Z})$. By comparing \eqref{eq: congruence subgroup I} with \eqref{eq: nu_1 n_1 M i}, we conclude
\begin{equation}
    \begin{pmatrix}
        1 & k \\
        0 & n
    \end{pmatrix}M
    =\gamma
    \begin{pmatrix}
        1 & \mu_2\\
        0 & m_2
    \end{pmatrix}
\end{equation}
so
\begin{equation}
    M=\frac{1}{n}\begin{pmatrix}
        n & -k \\
        0 & 1
    \end{pmatrix}\gamma
    \begin{pmatrix}
        1 & \mu_2 \\
        0 & m_2
    \end{pmatrix}
    =
    \begin{pmatrix}
        * & * \\
        \frac{c}{n} & *
    \end{pmatrix}.
\end{equation}
Since the entries of $M$ are integers, then $n\mid c$, which means $\gamma\in\Gamma_0(n)$.

(ii) If both $n$ and $\frac{D-\nu^2}{n}$ are even, then $m_j$ is even, and since then $\frac{D-\mu_j^2}{m_j}$ is odd, we have $m_j \equiv 0 \pmod {2n}$.
Denote $J_0=(2,1+\sqrt{D})$ a non-invertible ideal of $\mathbb{Z}[\sqrt{D}]$. Working in $\mathbb{Z}[\sqrt{D}]$, we obtain
\begin{equation}
    \begin{aligned}
        I_jJ_0&=(2m_j,2\mu_j+2\sqrt{D},m_j+m_j\sqrt{D},\mu_j+D+(\mu_j+1)\sqrt{D})\\
        &=(2m_j,2\mu_j+2\sqrt{D},m_j+m_j\sqrt{D},D-\mu_j^2)\\
        &=(2m_j,2\mu_j+2\sqrt{D},m_j+m_j\sqrt{D},m_j)\\
        &=2(m_j/2,\mu_j+\sqrt{D}) 
    \end{aligned}
\end{equation}
where the third equality follows by $D-\mu_j^2=r_jm_j$ with odd integer $r_j$. Since $4\mid m_j$, then both $\frac{m_j}{2}$ and $\frac{D-\mu_j^2}{m_j/2}$ are even. Let $J_j$ be the ideal of $\mathbb{Z}[\tfrac{1+\sqrt{D}}{2}]$ corresponding to the root $\mu_j \bmod \frac{m_j}{2}$. Since
\begin{equation}
  \label{eq:Jjbasis}
    \begin{pmatrix}
        1 & \frac{\mu_j-1}{2} \\
        0 & \frac{m_j}{4}
    \end{pmatrix}
    \begin{pmatrix}
        \frac{1+\sqrt{D}}{2} & \frac{1-\sqrt{D}}{2} \\
        1 & 1
    \end{pmatrix}
    =\frac{1}{2}
    \begin{pmatrix}
        1 & \mu_j \\
        0 & \frac{m_j}{2}
    \end{pmatrix}
    \begin{pmatrix}
        \sqrt{D} & -\sqrt{D} \\
        1 & 1
    \end{pmatrix},
\end{equation}
we have $J_j=\frac{1}{4}I_jJ_0$. 
Let $J$ be the ideal of $\mathbb{Z}[\frac{1+\sqrt{D}}{2}]$ corresponding to the root $\nu \bmod n$. Due to the congruence conditions modulo $n$, we have 
\begin{equation}
    \begin{pmatrix}
        1 & \frac{\mu_j-1}{2} \\
        0 & \frac{m_j}{4}
    \end{pmatrix}
    =
    \begin{pmatrix}
        1 & k_j \\
        0 & \frac{n_j}{2}
    \end{pmatrix}
    \begin{pmatrix}
        1 & \frac{\nu-1}{2} \\
        0 & \frac{n}{2}
    \end{pmatrix},
\end{equation}
where $k_j\in\mathbb{Z}$ and $n_j=\frac{m_j}{n}$.
This implies $J_j\subseteq J$ so $J_jJ^{-1}\subseteq \mathbb{Z}[\frac{1+\sqrt{D}}{2}]$. Let $\nu_1 \bmod n_1$ correspond to the ideal $J_1J^{-1}$. Since $J_1\subseteq J_1J^{-1}$, we have
\begin{equation}
    \begin{pmatrix}
        1 & \frac{\mu_1-1}{2} \\
        0 & \frac{m_1}{4}
    \end{pmatrix}
    =
    \begin{pmatrix}
        1 & k \\
        0 & \frac{n}{2}
    \end{pmatrix}
    \begin{pmatrix}
        1 & \frac{\nu_1-1}{2} \\
        0 & \frac{n_1}{2}
    \end{pmatrix}
\end{equation}
for some $k\in\mathbb{Z}$. As $\xi I_1=I_2$ and $J_j=\frac{1}{4}I_jJ_0$, we have $\xi J_1J^{-1}=J_2J^{-1}\subseteq\mathbb{Z}[\frac{1+\sqrt{D}}{2}]$, and so 
\begin{equation}
    \begin{pmatrix}
        1 & \frac{\nu_1-1}{2} \\
        0 & \frac{n_1}{2}
    \end{pmatrix}
    \begin{pmatrix}
        \frac{1+\sqrt{D}}{2} & \frac{1-\sqrt{D}}{2}\\
        1 & 1
    \end{pmatrix}  
    \begin{pmatrix}
        \xi & 0 \\
        0 & \overline{\xi}
    \end{pmatrix}
    =M
    \begin{pmatrix}
        \frac{1+\sqrt{D}}{2} & \frac{1-\sqrt{D}}{2}\\
        1 & 1
    \end{pmatrix}.
\label{eq: nu_1 n_1 I ii}\end{equation}
From a calculation similar to \eqref{eq:Jjbasis}, the equation \eqref{eq: congruence subgroup I} implies
\begin{equation}
    \begin{pmatrix}
        1 & \frac{\mu_1-1}{2} \\
        0 & \frac{m_1}{2}
    \end{pmatrix}
    \begin{pmatrix}
        \frac{1+\sqrt{D}}{2} & \frac{1-\sqrt{D}}{2}\\
        1 & 1
    \end{pmatrix}
    \begin{pmatrix}
        \xi & 0 \\
        0 & \overline{\xi}
    \end{pmatrix}
    =\gamma
    \begin{pmatrix}
        1 & \frac{\mu_2-1}{2} \\
        0 & \frac{m_2}{2}
    \end{pmatrix}
    \begin{pmatrix}
        \frac{1+\sqrt{D}}{2} & \frac{1-\sqrt{D}}{2}\\
        1 & 1
    \end{pmatrix}
\end{equation}
so this yields
\begin{equation}
    \begin{pmatrix}
        1 & \frac{\mu_1-1}{2} \\
        0 & \frac{m_1}{4}
    \end{pmatrix}\!
    \begin{pmatrix}
        \frac{1+\sqrt{D}}{2} & \frac{1-\sqrt{D}}{2}\\
        1 & 1
    \end{pmatrix}\!
    \begin{pmatrix}
        \xi & 0 \\
        0 & \overline{\xi}
    \end{pmatrix}
    =
    \begin{pmatrix}
        1 & 0 \\
        0 & \frac{1}{2}
    \end{pmatrix}\gamma
    \begin{pmatrix}
        1 & \frac{\mu_2-1}{2} \\
        0 & \frac{m_2}{2}
    \end{pmatrix}\!
    \begin{pmatrix}
        \frac{1+\sqrt{D}}{2} & \frac{1-\sqrt{D}}{2}\\
        1 & 1
    \end{pmatrix}
\label{eq: mu_1 m_1 I ii}\end{equation}
for some integer matrix $M\in\mathrm{M}_2(\mathbb{Z})$. By comparing \eqref{eq: nu_1 n_1 I ii} with \eqref{eq: mu_1 m_1 I ii}, we conclude
\begin{equation}
    \begin{pmatrix}
        1 & k \\
        0 & \frac{n}{2}
    \end{pmatrix}M
    =
    \begin{pmatrix}
        1 & 0 \\
        0 & \frac{1}{2}
    \end{pmatrix}\gamma
    \begin{pmatrix}
        1 & \frac{\mu_2-1}{2}\\
        0 & \frac{m_2}{2}
    \end{pmatrix}
\end{equation}
so
\begin{equation}
    M=\frac{2}{n}\begin{pmatrix}
        \frac{n}{2} & -k \\
        0 & 1
    \end{pmatrix}
    \begin{pmatrix}
        1 & 0 \\
        0 & \frac{1}{2}
    \end{pmatrix}
    \gamma
    \begin{pmatrix}
        1 & \frac{\mu_2-1}{2}\\
        0 & \frac{m_2}{2}
    \end{pmatrix}
    =
    \begin{pmatrix}
        * & * \\
        \frac{c}{n} & *
    \end{pmatrix}.
\end{equation}
As before, we have $n\mid c$ also and $\gamma\in\Gamma_0(n)$.
\end{proof}





\begin{lm}
    Let $D\equiv 1 \pmod 4$ and let $m_1, m_2$ be even. Let $\mu_1 \bmod m_1$ and $\mu_2 \bmod m_2$ satisfy $\mu_1^2\equiv D \pmod {m_1}$ and $\mu_2^2\equiv D \pmod {m_2}$ with both $\frac{D-\mu_1^2}{m_1}$ and $\frac{D-\mu_2^2}{m_2}$ even. Suppose that there is a root $\nu \bmod n$ of $\nu^2\equiv D \pmod n$ such that $m_1 \equiv m_2 \equiv 0 \pmod n$ and $\mu_1\equiv \mu_2\equiv \nu \pmod n$, and suppose that there exists $\gamma\in\SL(2,\mathbb{Z})$ such that
    \begin{equation}
        \begin{pmatrix}
            1 & \frac{\mu_1-1}{2} \\
            0 & \frac{m_1}{2}
        \end{pmatrix}
        \begin{pmatrix}
            \frac{1+\sqrt{D}}{2} & \frac{1-\sqrt{D}}{2}\\
            1 & 1
        \end{pmatrix}
        \begin{pmatrix}
            \zeta & 0 \\
            0 & \overline{\zeta}
        \end{pmatrix}
        =\gamma
        \begin{pmatrix}
            1 & \frac{\mu_2-1}{2} \\
            0 & \frac{m_2}{2}
        \end{pmatrix}
        \begin{pmatrix}
            \frac{1+\sqrt{D}}{2} & \frac{1-\sqrt{D}}{2}\\
            1 & 1
        \end{pmatrix}  
    \label{eq: congruence subgroup J}\end{equation}
    for some totally positive $\zeta$. (i) If $n$ is odd, then $\gamma\in\Gamma_0(n)$. (ii) If $n$ is even, then $\gamma\in\Gamma_0\!\left(\frac{n}{2}\right)$.
\label{lm: congruence subgroup J}\end{lm}

\begin{proof}
  By Proposition \ref{prp:general m equiv 2 mod 4}, we may find $J_1$ and $J_2$ be two ideals of $\mathbb{Z}[\frac{1+\sqrt{D}}{2}]$ respectively corresponding to the roots $\mu_1 \bmod {m_1}$ and $\mu_2 \bmod {m_2}$.
  We note that \eqref{eq: congruence subgroup J} implies $\zeta J_1=J_2$.

(i) If $n$ is odd, then we may have an ideal $I$ of $\mathbb{Z}[\sqrt{D}]$ corresponding to the root $\nu \bmod n$ by Proposition \ref{prp:general m not equiv 2 mod 4}. Since
\begin{equation}
    2\begin{pmatrix}
        1 & \frac{\mu_j-1}{2} \\
        0 & \frac{m_j}{2}
    \end{pmatrix}
    \begin{pmatrix}
        \frac{1+\sqrt{D}}{2} & \frac{1-\sqrt{D}}{2}\\
        1 & 1
    \end{pmatrix}
    =
    \begin{pmatrix}
        1 & \mu_j \\
        0 & m_j
    \end{pmatrix}
    \begin{pmatrix}
        \sqrt{D} & -\sqrt{D}\\
        1 & 1
    \end{pmatrix},    
\end{equation}
we have $2J_j$ is a (non-invertible) ideal of $\mathbb{Z}[\sqrt{D}]$. Due to the congruence conditions modulo $n$, we may write
\begin{equation}
    \begin{pmatrix}
        1 & \mu_j \\
        0 & m_j
    \end{pmatrix}
    =
    \begin{pmatrix}
        1 & k_j \\
        0 & n_j
    \end{pmatrix}
    \begin{pmatrix}
        1 & \nu \\
        0 & n
    \end{pmatrix}
\end{equation}
where $n_j=\frac{m_j}{n}\in\mathbb{Z}$ and $k_j\in\mathbb{Z}$, and so $2J_j \subseteq I$. Let $\nu_1 \bmod n_1$ correspond to the ideal $2J_1I^{-1}\subseteq\mathbb{Z}[\sqrt{D}]$. Then 
\begin{equation}
    \begin{pmatrix}
        1 & \mu_1 \\
        0 & m_1
    \end{pmatrix}
    =
    \begin{pmatrix}
        1 & k \\
        0 & n
    \end{pmatrix}
    \begin{pmatrix}
        1 & \nu_1 \\
        0 & n_1
    \end{pmatrix}
\end{equation}
for some $k\in\mathbb{Z}$. Since $\zeta(2J_1)I^{-1}=2J_2I^{-1}\subseteq\mathbb{Z}[\sqrt{D}]$, then
\begin{equation}
    \begin{pmatrix}
        1 & \nu_1 \\
        0 & n_1
    \end{pmatrix}
    \begin{pmatrix}
        \sqrt{D} & -\sqrt{D}\\
        1 & 1
    \end{pmatrix}  
    \begin{pmatrix}
        \zeta & 0 \\
        0 & \overline{\zeta}
    \end{pmatrix}
    =M
    \begin{pmatrix}
        \sqrt{D} & -\sqrt{D}\\
        1 & 1
    \end{pmatrix}  
\label{eq: (nu_1-1)/2 n_1/2 M i}\end{equation}
for some integer matrix $M\in\mathrm{M}_2(\mathbb{Z})$. By comparing \eqref{eq: congruence subgroup J} with \eqref{eq: (nu_1-1)/2 n_1/2 M i}, we conclude
\begin{equation}
    \begin{pmatrix}
        1 & k \\
        0 & n
    \end{pmatrix}M
    \begin{pmatrix}
        1 & 1 \\
        0 & 2
    \end{pmatrix}^{-1}
    =\gamma
    \begin{pmatrix}
        1 & \frac{\mu_2-1}{2}\\
        0 & \frac{m_2}{2} 
    \end{pmatrix},
\end{equation}
so
\begin{equation}
    M=\frac{1}{n}\begin{pmatrix}
        n & -k \\
        0 & 1
    \end{pmatrix}\gamma
    \begin{pmatrix}
        1 & \frac{\mu_2-1}{2}\\
        0 & \frac{m_2}{2} 
    \end{pmatrix}
    \begin{pmatrix}
        1 & 1 \\
        0 & 2
    \end{pmatrix}
    =
    \begin{pmatrix}
        * & * \\
        \frac{c}{n} & *
    \end{pmatrix}.
\end{equation}
As before, we conclude $\gamma\in\Gamma_0(n)$.

(ii) If $n$ is even, then let the ideal $J$ of $\mathbb{Z}[\frac{1+\sqrt{D}}{2}]$ correspond to the root $\nu \bmod n$ by Proposition \ref{prp:general m equiv 2 mod 4}. We now may write
\begin{equation}
    \begin{pmatrix}
        1 & \frac{\mu_j-1}{2} \\
        0 & \frac{m_j}{2}
    \end{pmatrix}
    =
    \begin{pmatrix}
        1 & k_j \\
        0 & \frac{n_j}{2}
    \end{pmatrix}
    \begin{pmatrix}
        1 & \frac{\nu-1}{2} \\
        0 & \frac{n}{2}
    \end{pmatrix},
\end{equation}
where $n_j=\frac{2m_j}{n}\in\mathbb{Z}$ and $k_j\in\mathbb{Z}$. This implies $J_j\subseteq J$ so $J_jJ^{-1}\subseteq\mathbb{Z}[\frac{1+\sqrt{D}}{2}]$ for $j=1,2$. Let $\nu_1 \bmod n_1$ be the root corresponding to the ideal $J_1J^{-1}$. Then 
\begin{equation}
    \begin{pmatrix}
        1 & \frac{\mu_1-1}{2} \\
        0 & \frac{m_1}{2}
    \end{pmatrix}
    =
    \begin{pmatrix}
        1 & k \\
        0 & \frac{n}{2}
    \end{pmatrix}
    \begin{pmatrix}
        1 & \frac{\nu_1-1}{2} \\
        0 & \frac{n_1}{2}
    \end{pmatrix}
\end{equation}
for some $k\in\mathbb{Z}$. Since $\zeta J_1J^{-1}=J_2J^{-1}\subseteq\mathbb{Z}[\frac{1+\sqrt{D}}{2}]$,
\begin{equation}
    \begin{pmatrix}
        1 & \frac{\nu_1-1}{2} \\
        0 & \frac{n_1}{2}
    \end{pmatrix}
    \begin{pmatrix}
        \frac{1+\sqrt{D}}{2} & \frac{1-\sqrt{D}}{2}\\
        1 & 1
    \end{pmatrix}  
    \begin{pmatrix}
        \zeta & 0 \\
        0 & \overline{\zeta}
    \end{pmatrix}
    =M
    \begin{pmatrix}
        \frac{1+\sqrt{D}}{2} & \frac{1-\sqrt{D}}{2}\\
        1 & 1
    \end{pmatrix}  
\label{eq: nu_1-1/2 n_1/2 M ii}\end{equation}
holds for some integer matrix $M\in\mathrm{M}_2(\mathbb{Z})$. By comparing \eqref{eq: congruence subgroup J} with \eqref{eq: nu_1-1/2 n_1/2 M ii}, we conclude
\begin{equation}
    \begin{pmatrix}
        1 & k \\
        0 & \frac{n}{2}
    \end{pmatrix}M
    =\gamma
    \begin{pmatrix}
        1 & \frac{\mu_2-1}{2}\\
        0 & \frac{m_2}{2} 
    \end{pmatrix},
\end{equation}
so
\begin{equation}
    M=\frac{2}{n}\begin{pmatrix}
        \frac{n}{2} & -k \\
        0 & 1
    \end{pmatrix}\gamma
    \begin{pmatrix}
        1 & \frac{\mu_2-1}{2}\\
        0 & \frac{m_2}{2} 
    \end{pmatrix}
    =
    \begin{pmatrix}
        * & * \\
        \frac{2c}{n} & *
    \end{pmatrix}.
\end{equation}
As $\frac{n}{2}\mid c$, the claim $\gamma\in\Gamma_0\!\left(\frac{n}{2}\right)$ has proven.
\end{proof}

Fixed a root $\nu \bmod n$, we have now established that the tops $\frac{\mu}{m} + \mathrm{i} \frac{\sqrt{D}}{m}$ with $\mu \equiv 0 \pmod n$ and $\mu \equiv \nu \pmod n$ are exactly obtained from the $\Gamma_0(n)$-orbits of the geodesics $\gamma_{1;k} \bm{c}_{I_k}$ for $1\leq k \leq h_1$ if $m$ or $\frac{D-\mu^2}{m}$ is odd; otherwise, they come from either the $\Gamma_0(n)$-orbits (when $n$ is odd) or the $\Gamma_0(\frac{n}{2})$-orbits (when $n$ is even) of the geodesics $\gamma_{2;l} \bm{c}_{J_l}$ for $1 \leq l \leq h_2$.
We now consider the length of these geodesics projected to $\Gamma_0(n) \backslash \mathbb{H}$ or $\Gamma_0(\frac{n}{2}) \backslash \mathbb{H}$.
This amounts to understanding the stabilizers of these geodesics in $\Gamma_0(n)$ or $\Gamma_0(\frac{n}{2})$, and we begin by recording a couple of basic facts about the relation between the totally positive fundamental units $\varepsilon_1$ and $\varepsilon_2$ of $\mathbb{Z}[\sqrt{D}]$ and $\mathbb{Z}[\tfrac{1 + \sqrt{D}}{2}]$.

\begin{prp}
    Let $\varepsilon_1$ and $\varepsilon_2$ be totally positive fundamental units of $\mathbb{Z}[\sqrt{D}]$ and $\mathbb{Z}[\frac{1+\sqrt{D}}{2}]$ respectively.
    \begin{enumerate}[(i)]
    \item If $D\equiv 1\pmod 8$, then $\varepsilon_1=\varepsilon_2$.
        \item If $D\equiv 5\pmod 8$, then either $\varepsilon_1=\varepsilon_2$ or $\varepsilon_1=\varepsilon_2^3$.
    \end{enumerate}
\label{prp: epsilon1 epsilon2}\end{prp}

\begin{proof}
  Set $\varepsilon_2=a+b\frac{1+\sqrt{D}}{2}$ with $a,b\in\mathbb{Z}$. 

  (i) We have $\N(\varepsilon_2)=a^2+ab-b^2\frac{D-1}{4}=1$, and since $D\equiv 1\pmod 8$, we know $\frac{D-1}{4}$ is even.
  If $b$ is odd, then $a(a+b)$ is even and so is $\N(\varepsilon_2)$, a contradiction. Hence, $b$ is even, so $\varepsilon_2 \in \mathbb{Z}[\sqrt{D}]$ as required.

  (ii) Now, in the case that $D\equiv 5\pmod 8$, if $b$ is even, then $\varepsilon_1=\varepsilon_2$ as before.
  Suppose that $b$ is odd, so $\varepsilon_2\not\in\mathbb{Z}[\sqrt{D}]$, that is, $\varepsilon_1\neq\varepsilon_2$. As $a^2+ab-b^2\frac{D-1}{4}=1$, 
\begin{equation}
\begin{aligned}
    \varepsilon_2^2&=a^2+ab(1+\sqrt{D})+b^2\tfrac{D-1}{4}+b^2\tfrac{1+\sqrt{D}}{2} \\
    &=1-ab+b^2\tfrac{D-1}{2}+b(b+2a)\tfrac{1+\sqrt{D}}{2}.
\end{aligned}    
\end{equation}
Since $b(b+2a)$ is odd, $\varepsilon_2^2\not\in\mathbb{Z}[\sqrt{D}]$ so $\varepsilon_1\neq\varepsilon_2^2$. Now,
\begin{equation}
\begin{aligned}
    \varepsilon_2^3&=a-a^2b+ab^2(D-1)+b^3\tfrac{D-1}{4}+b\left(1+2a^2+2ab+b^2\tfrac{D+1}{2}\right)\tfrac{1+\sqrt{D}}{2}\\
    &=(a+ab^2D-b)+b(3+b^2D)\tfrac{1+\sqrt{D}}{2}.
\end{aligned}
\end{equation}
As $b$ is odd, $b(3+b^2D)$ is even. This implies $\varepsilon_2^3\in\mathbb{Z}[\sqrt{D}]$ and so $\varepsilon_1=\varepsilon_2^3$.
\end{proof}

We now turn to the stabilizers of the geodesics $\gamma_{1;k} \bm{c}_{I_k}$ and $\gamma_{2;l}\bm{c}_{I_l}$ in $\Gamma_0(n)$ or $\Gamma_0(\frac{n}{2})$.
These will have a generator of the form
\begin{equation}
  \label{eq:generator1}
  \gamma_{1;k} \mathfrak{B}_k
\begin{pmatrix}
  \varepsilon_1^{j_1} & 0 \\
  0 & \varepsilon_1^{-j_1}
\end{pmatrix}
\mathfrak{B}_k \gamma_{1;k}^{_1},\ \gamma_{2;l} \mathfrak{D}_l
\begin{pmatrix}
  \varepsilon_2^{j_2} & 0 \\
  0 & \varepsilon_2^{-j_2}
\end{pmatrix}
\mathfrak{D}_l^{-1} \gamma_{2;l}^{-1},
\end{equation}
where $j_1$ and $j_2$ are the smallest positive integers so that these matrices are in $\Gamma_0(n)$ or $\Gamma_0(\frac{n}{2})$.
We note that the length of the projections of these geodesics to $\Gamma_0(n) \backslash \mathbb{H}$ or $\Gamma_0(\frac{n}{2}) \backslash \mathbb{H}$ is then $2 j_1 \log \varepsilon_1$, $2 j_2 \log \varepsilon_2$.
The following lemmas show that $j_1$ and $j_2$ are equal to $1$.

\begin{lm}
    For any $k$, we have 
    \begin{equation}
      \label{eq:generator2}
        \gamma_{1;k}\mathfrak{B}_k
        \begin{pmatrix}
            \varepsilon_1 & 0 \\
            0 & \varepsilon_1^{-1}
        \end{pmatrix}
        \mathfrak{B}_k^{-1}\gamma_{1;k}^{-1}\in\Gamma_0(n).
      \end{equation}
      Similarly, for any $l$, we have
      \begin{equation}
        \label{eq:generator3}
        \gamma_{2;l} \mathfrak{D}_l
\begin{pmatrix}
  \varepsilon_2 & 0 \\
  0 & \varepsilon_2^{-1}
\end{pmatrix}
\mathfrak{D}_l^{-1} \gamma_{2;l}^{-1} \in \Gamma_0(n)
\end{equation}
if $n$ is odd; otherwise (if $n$ is even),
\begin{equation}
  \label{eq:generator4}
  \gamma_{2;l} \mathfrak{D}_l
\begin{pmatrix}
  \varepsilon_2 & 0 \\
  0 & \varepsilon_2^{-1}
\end{pmatrix}
\mathfrak{D}_l^{-1} \gamma_{2;l}^{-1} \in \Gamma_0(\tfrac{n}{2}).
\end{equation} 
\label{lm: gamma_0 I}\end{lm}

\begin{proof}
  Let $\gamma$ denote the matrix in \eqref{eq:generator2}.
  From the definition of $\gamma_{1;k}$, we have that
\begin{equation}
    \gamma_{1;k}\mathfrak{B}_k
    \begin{pmatrix}
        \xi & 0 \\
        0 & \overline{\xi}
    \end{pmatrix}
    =
    \begin{pmatrix}
        1 & \mu \\
        0 & m
    \end{pmatrix}
    \begin{pmatrix}
        \sqrt{D} & -\sqrt{D} \\
        1 & 1
    \end{pmatrix}
  \end{equation}
  for some totally positive $\xi$ and root $\mu \bmod m$ satisfying $m \equiv 0 \pmod n$ and $\mu \equiv \nu \pmod n$.
  Therefore
\begin{equation}
  \gamma
  \begin{pmatrix}
        1 & \mu \\
        0 & m
    \end{pmatrix}
    \begin{pmatrix}
        \sqrt{D} & -\sqrt{D} \\
        1 & 1
    \end{pmatrix}
    =
    \begin{pmatrix}
        1 & \mu \\
        0 & m
    \end{pmatrix}
    \begin{pmatrix}
        \sqrt{D} & -\sqrt{D} \\
        1 & 1
    \end{pmatrix}
    \begin{pmatrix}
        \varepsilon_1 & 0 \\
        0 & \varepsilon_1^{-1}
    \end{pmatrix}.
  \end{equation}
  Applying Lemma \ref{lm: congruence subgroup I} in the case that $m_1 = m_2 = m$ and $\mu_1 = \mu_2 = \mu$, we conclude that $\gamma \in \Gamma_0(n)$.

  The other claims in the lemma follow similarly using Lemma \ref{lm: congruence subgroup J}. 
\end{proof}

The inclusion \eqref{eq:generator4} raises the question of whether or not the matrix is in fact in $\Gamma_0(n)$ instead of just $\Gamma_0(\frac{n}{2})$.
As the answer is useful in what follows, we record the following lemma.

\begin{lm}
  \label{lm:Gamma0n}
  Suppose that $n$ is even.
  Then we have $\gamma_{2;l} \mathfrak{D}_l
\begin{pmatrix}
  \varepsilon_2 & 0 \\
  0 & \varepsilon_2^{-1}
\end{pmatrix}
\mathfrak{D}_l^{-1} \gamma_{2;l}^{-1} \in \Gamma_0(n)$ if and only if $\varepsilon_1 = \varepsilon_2$.
\end{lm}

Before proving this lemma, we remark that in view of Proposition \ref{prp: epsilon1 epsilon2}, $\varepsilon_1 = \varepsilon_2^3$ implies $D \equiv 5 \pmod 8$, which, since $n$ is even, so we may then assume $n \equiv 2 \pmod 4$.

\begin{proof}
  We let $\gamma$ denote the matrix in question, so
  \begin{equation}
    \label{eq:gammaeq}
    \gamma
    \begin{pmatrix}
      1 & \frac{\mu -1}{2} \\
      0 & \frac{m}{2}
    \end{pmatrix}
    \begin{pmatrix}
      \tfrac{1 + \sqrt{D}}{2} & \tfrac{1 - \sqrt{D}}{2} \\
      1 & 1
    \end{pmatrix}
    =
    \begin{pmatrix}
      1 & \frac{\mu -1}{2} \\
      0 & \frac{m}{2}
    \end{pmatrix}
    \begin{pmatrix}
      \tfrac{1 + \sqrt{D}}{2} & \tfrac{1 - \sqrt{D}}{2} \\
      1 & 1
    \end{pmatrix}
    \begin{pmatrix}
      \varepsilon_2 & 0 \\
      0 & \varepsilon_2^{-1}
    \end{pmatrix}
    ,
  \end{equation}
where $m \equiv 0 \pmod n$ and $\mu \equiv \nu \pmod n$. 
When $\varepsilon_1 = \varepsilon_2^3$, we must have $m \equiv n \pmod {2n}$. 
  We write $\varepsilon_2 = a + b \sqrt{D}$ and recall that $\varepsilon_1 = \varepsilon_2$ if $b$ is even and $\varepsilon_1 = \varepsilon_2^3$ if $b$ is odd.
  We have
  \begin{equation}
    \label{eq:epsilonconj}
    \begin{pmatrix}
      \tfrac{1 + \sqrt{D}}{2} & \tfrac{1 - \sqrt{D}}{2} \\
      1 & 1
    \end{pmatrix}
    \begin{pmatrix}
      \varepsilon_2 & 0 \\
      0 & \varepsilon_2^{-1}
    \end{pmatrix}
    \begin{pmatrix}
      \tfrac{1 + \sqrt{D}}{2} & \tfrac{1 - \sqrt{D}}{2} \\
      1 & 1
    \end{pmatrix}
    ^{-1} =
    \begin{pmatrix}
      * & * \\
      b & * 
    \end{pmatrix}
    ,
  \end{equation}
  so $\gamma =
    \begin{pmatrix}
      * & * \\
      \frac{bm}{2} & *
    \end{pmatrix}$
  and the lemma follows. 
\end{proof}

Putting the results of this section so far together, we obtain the following theorems.
Here we use the notation $\Gamma_{1;k}'$ for the stabilizer in $\Gamma_0(n)$ of the geodesic $\gamma_{1; k} \bm{c}_{I_k}$, so $\Gamma_{1;k}' = \gamma_{1;k} \Gamma_{1;k} \gamma_{1;k}^{-1}$, recalling the notation $\Gamma_{1;k}$ of section \ref{sec:geodesics}.
Similarly, we use $\Gamma_{2;l}' = \gamma_{2;l} \Gamma_{2;l} \gamma_{2;l}^{-1}$ to denote the stabilizer of $\gamma_{2;l} \bm{c}_{J_l}$ in $\Gamma_0(n)$ if $n$ is odd and in $\Gamma_0(\frac{n}{2})$ if $n$ is even. 

\begin{thm}
    Let $D>0$ be a square-free integer satisfying $D\equiv 1 \pmod 4$. Given a fixed positive integer $n$ and $\nu \bmod n$ with $\nu^2\equiv D \pmod n$, there exists a finite set of closed geodesics $\{\gamma_{1;1}\bm{c}_{I_1},\cdots,\gamma_{1;h_1}\bm{c}_{I_{h_1}}\}$ all with the same length in $\Gamma_0(n)\backslash\mathbb{H}$, having the following properties:
    \begin{enumerate}[(i)]
        \item For any positive integer $m$ and $\mu \bmod m$ satisfying $\mu^2\equiv D \pmod m$ with $m\equiv 0\pmod n$ and $\mu\equiv\nu\pmod n$ and with either $m$ or $\frac{D-\mu^2}{m}$ is odd, then there exists a unique $k$ and a double coset $\Gamma_\infty\gamma\Gamma_{1;k}'\in\Gamma_\infty\backslash\Gamma_0(n)\slash\Gamma_{1;k}'$ such that 
          \begin{equation}
            \label{eq:thmI}
            z_{\gamma\gamma_{1;k}\bm{c}_{I_k}}\equiv\frac{\mu}{m}+\mathrm{i}\frac{\sqrt{D}}{m} \pmod {\Gamma_\infty}.
        \end{equation}
        \item Conversely, given $k$ and a double coset $\Gamma_\infty\gamma\Gamma_{1;k}'\in\Gamma_\infty\backslash\Gamma_0(n)\slash\Gamma_{1;k}'$ with positively oriented $\gamma\gamma_{1;k}\bm{c}_{I_k}$, there exists a unique positive integer $m$ and a residue class $\mu \bmod m$ satisfying \eqref{eq:thmI} with $\mu^2 \equiv D \pmod m$, $m\equiv 0 \pmod n$, $\mu \equiv \nu \pmod n$ and either $m$ or $\frac{D-\mu^2}{m}$ being odd.
    \end{enumerate}
    \label{thm: congruence subgroups I}
  \end{thm}


\begin{thm}
    Let $D>0$ be a square-free integer satisfying $D\equiv 1 \pmod 4$. Given a fixed positive integer $n$ and $\nu \bmod n$ with $\nu^2\equiv D \pmod n$, there exists a finite set of closed geodesics $\{\gamma_{2;1}\bm{c}_{J_1},\cdots,\gamma_{2;h_2}\bm{c}_{J_{h_2}}\}$, having the following properties:
    \begin{enumerate}[(i)]
        \item For any positive integer $m$ and $\mu \bmod m$ satisfying $\mu^2\equiv D \pmod m$ with $m\equiv 0\pmod n$ and $\mu\equiv\nu\pmod n$ and with $m$ and $\frac{D-\mu^2}{m}$ both even, then there exists a unique $l$ and a double coset $\Gamma_\infty\gamma\Gamma_{2;l}'\in\Gamma_\infty\backslash\Gamma_0(n)\slash\Gamma_{2;l}'$ if $n$ is odd or in ${\Gamma_\infty\backslash\Gamma_0(\frac{n}{2})\slash\Gamma_{2;l}'}$ if $n$ is even such that 
          \begin{equation}
            \label{eq:thmJ}
            z_{\gamma\gamma_{2;l}\bm{c}_{J_l}}\equiv\frac{\mu}{m}+\mathrm{i}\frac{\sqrt{D}}{m} \pmod {\Gamma_\infty}.
        \end{equation}
        \item Conversely, given $l$ and a double coset $\Gamma_\infty\gamma\Gamma_{2;l}'\in\Gamma_\infty\backslash\Gamma_0(n)\slash\Gamma_{2;l}'$ if $n$ is odd or $\Gamma_\infty\backslash\Gamma_0(\frac{n}{2})\slash\Gamma_{2;l}'$ if $n$ is even such that $\gamma\gamma_{2;l}\bm{c}_{J_l}$ is positively oriented, there exists a unique positive integer $m$ and a residue class $\mu \bmod m$ satisfying \eqref{eq:thmJ} with $\mu^2 \equiv D \pmod m$, $m\equiv 0 \pmod n$, $\mu \equiv \nu \pmod n$ and both $m$ and $\frac{D-\mu^2}{m}$ being even.
    \end{enumerate}
    \label{thm: congruence subgroups J}\end{thm}


To combine Theorems \ref{thm: congruence subgroups I} and \ref{thm: congruence subgroups J}, we break the $\Gamma_0(\frac{n}{2})$-orbits appearing in Theorem \ref{thm: congruence subgroups J} when $n$ is even into a union of orbits under the subgroup $\Gamma_0(n)$.
In this way, we obtain a uniform acting group, $\Gamma_0(n)$, as stated in Theorem \ref{thm: congruence subgroups} and as needed to apply the results of \cite{misc:Marklof-Welsh}. 

We first look at the index $\left[\Gamma_0(\frac{n}{2}):\Gamma_0(n)\right]$.
The well-known formula
\begin{equation}
    \left[\SL(2,\mathbb{Z}):\Gamma_0(n)\right]=n\prod_{p\mid n}\left(1+\frac{1}{p}\right)
\end{equation}
implies
\begin{equation}
    \left[\Gamma_0\!\left(\tfrac{n}{2}\right):\Gamma_0(n)\right]=
    \begin{cases}
        2 & \text{if } n \equiv 0 \pmod 4 \\
        3 & \text{if } n \equiv 2 \pmod 4.
    \end{cases}
\end{equation} 
In the case that $n\equiv 0 \pmod 4$, we have distinct cosets $\Gamma_0(n)$ and $\Gamma_0(n)\gamma_0'$, say,
and in the case
that $n \equiv 2 \pmod 4$, we include $\Gamma_0(n)\gamma_0''$ as the third coset.

Given the geodesic $\gamma_{2;l}\bm{c}_{J_l}$, one may ask whether the geodesics $\gamma_{2;l}\bm{c}_{J_l}, \gamma_0'\gamma_{2;l}\bm{c}_{J_l}$ and $\gamma_0''\gamma_{2;l}\bm{c}_{J_l}$ are $\Gamma_0(n)$-equivalent.
Geometrically, this is asking about the splitting behaviour of the geodesic $\gamma_{2;l}\bm{c}_{J_l}$ as it lifts from the surface $\Gamma_0(\frac{n}{2}) \backslash \mathbb{H}$ to the surface $\Gamma_0(n) \backslash \mathbb{H}$, see figures \ref{fig: cover 1} and \ref{fig: cover 2}.


\begin{figure}
    \centering
\begin{tikzpicture}
    \draw[->] (-8em,0.5em) -- (-8em,-0.5em);
    \draw (-8em,3em) ellipse (1em and 0.5em);
    \draw (-8em,1.5em) ellipse (1em and 0.5em);
    \draw (-8em,-1.5em) ellipse (1em and 0.5em);
    
    \node at (-16em,2.25em) {$\Gamma_0(n)\backslash\mathbb{H}$ :};
    \draw[->] (-16em,0.875em) -- (-16em,-0.125em);
    \node at (-16em,-1.5em) {$\Gamma_0(\frac{n}{2})\backslash\mathbb{H}$ :};

    \draw[->] (0em,0.5em) -- (0em,-0.5em);
    \draw[dashed] (0em,2.75em) ellipse (1em and 0.5em);
    \draw[dashed] (0em,1.75em) ellipse (1em and 0.5em);
    \draw[dashed] (0em,-1.5em) ellipse (1em and 0.5em);
    
    \node at (-8em,-3.5em) {(a) totally split};
    \node at (0em,-3.5em) {(b) inert};
\end{tikzpicture}
\caption{A mapping from $\Gamma_0(n)\backslash\mathbb{H}$ to $\Gamma_0(\frac{n}{2})\backslash\mathbb{H}$ when $n\equiv 0\pmod 4$. (a) Two single geodesics are projected to a single geodesic with length $2\log \varepsilon_2$. (b) A double geodesic is projected to a single geodesic,
}
\label{fig: cover 1}\end{figure}

Let $\gamma_s=\gamma_{2;l}\mathfrak{D}_l\begin{pmatrix} \varepsilon_2 & 0 \\ 0 & \varepsilon_2^{-1} \end{pmatrix}\mathfrak{D}_l^{-1}\gamma_{2;l}^{-1}$ be the generator of the stabilizer of $\gamma_{2;l}\bm{c}_{J_l}$ in $\Gamma_0(\frac{n}{2})$.
Considering $n\equiv 0\pmod 4$ first, we have $\gamma_{s;l}^2 \in \Gamma_0(n)$, and $\gamma_{2;l} \bm{c}_{J_l}$ is $\Gamma_0(n)$-equivalent to $\gamma_0' \gamma_{2;l} \bm{c}_{J_l}$ if and only if $\gamma_s \not\in \Gamma_0(n)$.
Geometrically, if $\gamma_s \in \Gamma_0(n)$, then the geodesic in $\Gamma_0(\frac{n}{2}) \backslash \mathbb{H}$ lifts to two geodesics in $\Gamma_0(n) \backslash \mathbb{H}$, case (a) in Figure \ref{fig: cover 1}, and if $\gamma_s \not\in \Gamma_0(n)$ then the geodesic lifts to a single geodesic, case (b) in Figure \ref{fig: cover 1}.
However, in this case ($n \equiv 0 \pmod 4$) we may assume $D \equiv 1 \pmod 8$, and so by Lemma \ref{lm:Gamma0n} and Proposition \ref{prp: epsilon1 epsilon2}, $\gamma_s \in \Gamma_0(n)$ and so case (b) does not occur. 

Now if $n \equiv 2 \pmod 4$, using Proposition \ref{prp: epsilon1 epsilon2} and the same calculation as in the proof of Lemma \ref{lm:Gamma0n} with $\varepsilon_2$ replaced by $\varepsilon_2^3$, we have $\gamma_s^3 \in \Gamma_0(n)$.
These calculations show in addition that either $\gamma_s \in \Gamma_0(n)$ or $\gamma_s, \gamma_s^2 \not\in \Gamma_0(n)$.
The first case corresponds to case (c) in Figure \ref{fig: cover 2} and the second corresponds to case (d), and we note that more complicated behaviour, like that illustrated in case (e), does not occur.
The cases (c) and (d) correspond to $\varepsilon_1 = \varepsilon_2$ and $\varepsilon_1 = \varepsilon_2^3$.

\begin{figure}
    \centering
\begin{tikzpicture}
    \draw[->] (-8em,0.5em) -- (-8em,-0.5em);
    \draw (-8em,4.5em) ellipse (1em and 0.5em);
    \draw (-8em,3em) ellipse (1em and 0.5em);
    \draw (-8em,1.5em) ellipse (1em and 0.5em);
    \draw (-8em,-1.5em) ellipse (1em and 0.5em);
    \node at (-16em,3em) {$\Gamma_0(n)\backslash\mathbb{H}$ :};
    \draw[->] (-16em,1.25em) -- (-16em,0.25em);
    \node at (-16em,-1.5em) {$\Gamma_0(\frac{n}{2})\backslash\mathbb{H}$ :};

    \draw[->] (0em,0.5em) -- (0em,-0.5em);
    \draw (0em,4em) ellipse (1em and 0.5em);
    \draw (0em,3em) ellipse (1em and 0.5em);
    \draw (0em,2em) ellipse (1em and 0.5em);
    \draw (0em,-1.5em) ellipse (1em and 0.5em);
    
    \draw[->] (8em,0.5em) -- (8em,-0.5em);
    \draw[dashed] (8em,4.5em) ellipse (1em and 0.5em);
    \draw[dashed] (8em,3em) ellipse (1em and 0.5em);
    \draw[dashed] (8em,2em) ellipse (1em and 0.5em);
    \draw[dashed] (8em,-1.5em) ellipse (1em and 0.5em);
    
    \node at (-8em,-3.5em) {(c) totally split};
    \node at (0em,-3.5em) {(d) inert};
    \node at (8em,-3.5em) {(e) split};
\end{tikzpicture}
\caption{A mapping from $\Gamma_0(n)\backslash\mathbb{H}$ to $\Gamma_0(\frac{n}{2})\backslash\mathbb{H}$ when $n\equiv 2\pmod 4$. (c) Three single geodesics are projected to a single geodesic with length $2\log\varepsilon_2$. (d) A triple geodesic with length $2\log\varepsilon_2^3$ is projected to a single geodesic. (e) A single geodesic and a double geodesic are projected to a single geodesic. 
}
\label{fig: cover 2}\end{figure}

We now combine Theorems \ref{thm: congruence subgroups I} and \ref{thm: congruence subgroups J} to prove our main result, Theorem \ref{thm: congruence subgroups}.

\begin{proof}[Proof of Theorem \ref{thm: congruence subgroups}] Define the number
\begin{equation}
    s=\begin{cases}
        1 & \text{if $D\equiv 5\pmod 8$ and $\varepsilon_1=\varepsilon_2^3$,} \\
        2 & \text{if $D\equiv 1\pmod 8$ and $n\equiv 0\pmod 4$,}\\
        3 & \text{otherwise.}
    \end{cases}
\end{equation}
Let $\{\bm{c}_1,\cdots,\bm{c}_{h}\}$ be a finite set of closed geodesics where $h=h_1+sh_2$. If $s=1$, then set
\begin{equation}
    \bm{c}_k=\begin{cases}
          \gamma_{1;k}\bm{c}_{I_k} & 1\leq k\leq h_1, \\
          \gamma_{2;k-h_1}\bm{c}_{J_{k-h_1}} & h_1+1\leq k \leq h_1+h_2.
    \end{cases}
\end{equation}
If $s=2$, then $\Gamma_0(\frac{n}{2})=\Gamma_0(n)\cup\Gamma_0(n)\gamma_0'$ for some $\gamma_0'\in\Gamma_0(\frac{n}{2})$ and set
\begin{equation}
    \bm{c}_k=\begin{cases}
          \gamma_{1;k}\bm{c}_{I_k} & 1\leq k\leq h_1, \\
          \gamma_{2;k-h_1}\bm{c}_{J_{k-h_1}} & h_1+1\leq k \leq h_1+h_2, \\
          \gamma_0'\gamma_{2;k-h_1-h_2}\bm{c}_{J_{k-h_1-h_2}} & h_1+h_2+1\leq k \leq h_1+2h_2.
    \end{cases}
\end{equation}
Otherwise $s=3$, we have $\Gamma_0(\frac{n}{2})=\Gamma_0(n)\cup\Gamma_0(n)\gamma_0'\cup\Gamma_0(n)\gamma_0''$ for some $\gamma_0', \gamma_0''\in\Gamma_0(\frac{n}{2})$ and set
\begin{equation}
    \bm{c}_k=\begin{cases}
          \gamma_{1;k}\bm{c}_{I_k} & 1\leq k\leq h_1, \\
          \gamma_{2;k-h_1}\bm{c}_{J_{k-h_1}} & h_1+1\leq k \leq h_1+h_2, \\
          \gamma_0'\gamma_{2;k-h_1-h_2}\bm{c}_{J_{k-h_1-h_2}} & h_1+h_2+1\leq k \leq h_1+2h_2, \\
          \gamma_0''\gamma_{2;k-h_1-2h_2}\bm{c}_{J_{k-h_1-2h_2}} & h_1+2h_2+1\leq k \leq h_1+3h_2.
    \end{cases}
\end{equation}
Define 
\begin{equation}
    \Gamma_k=\begin{cases}
        \Gamma_{1;k}' & 1\leq k \leq h_1, \\
        \Gamma_{2;k-h_1 -s'h_2}' & h_1 + s'h_2 +1\leq k \leq h_1+sh_2, 0\leq s' < s.
    \end{cases}
\end{equation}
The claims in Theorem \ref{thm: congruence subgroups} now follow from Theorems \ref{thm: congruence subgroups I} and \ref{thm: congruence subgroups J}.
\end{proof}

\bigskip

\section{Pair correlation density}
\label{sec:density}

We recall from theorem $4.10$ in \cite{misc:Marklof-Welsh} that the pair correlation density is given by the even and continuous function
\begin{equation}
    w(v)=\frac{1}{\mathrm{vol}_{\mathbb{H}}(\Gamma \backslash \mathbb{H})}\sum_{k,l=1}^h\sum_{\substack{\gamma\in\Gamma_{\bm{c}_k}\backslash\Gamma\slash\Gamma_{\bm{c}_l} \\ \gamma\bm{c}_l \neq \bm{c}_k, \overline{\bm{c}}_k}} \frac{1}{v^2} H_{\sgn(g_k^{-1}\gamma g_l(0))}(q,v/\kappa_\Gamma),
\label{eq: W(v)}\end{equation}
where $\kappa_\Gamma$ is the total length of the geodesics divided by $2\pi \mathrm{vol}_{\mathbb{H}}(\Gamma \backslash \mathbb{H})$ and $q = q(\gamma, l_1, l_2) = \frac{r + 1}{r-1}$ with $r$ the cross-ratio
\begin{equation}
  \label{eq:cross-ratio}
  r = \frac{ ((\gamma \bm{c}_{l_2})^+ - c_{l_1}^{-}) ((\gamma \bm{c}_{l_2})^- - \bm{c}_{l_1}^+)}{((\gamma \bm{c}_{l_2})^+ - \bm{c}_{l_1}^+)((\gamma \bm{c}_{l_2})^- - \bm{c}_{l_1}^-)},
\end{equation}
and
\begin{equation}
    H_+(q,v)=\begin{cases}
        0 & \text{if } q<-1 \\
        0 & \text{if } |q|<1 \text{ and } v<\sqrt{2-2q}\\
        h_q(s_1(q,v))-h_q(s_2(q,v)) & \text{if } |q|<1 \text{ and } v>\sqrt{2-2q}\\
        h_q(s_1(q,v))-h_q(-q+\sqrt{q^2-1}) & \text{if } q>1,
    \end{cases}
\end{equation}
and
\begin{equation}
    H_-(q,v)=\begin{cases}
        0 & \text{if } q<-1 \text{ and } |v|<\sqrt{2-2q} \\
        h_q(s_1(q,v))-h_q(s_2(q,v)) & \text{if } q<-1 \text{ and } |v|>\sqrt{2-2q}\\
        h_q(s_1(q,v))-h_q(s_2(q,v)) & \text{if } |q|<1 \text{ and } v<-\sqrt{2-2q}\\
        0 & \text{if } |q|<1 \text{ and } v>-\sqrt{2-2q} \\
        h_q(-q-\sqrt{q^2-1})-h_q(s_2(q,v)) & \text{if } q>1,
    \end{cases}
\end{equation}
with
\begin{equation}
    y(q,v)=\sqrt{v^2+q^2-1},
\end{equation}
\begin{equation}
    s_1(q,v)=\frac{-q+y(q,v)}{v+1}, \qquad s_2(q,v)=v-q-y(q,v),
\end{equation}
and
\begin{equation}
    h_q(s)=\log\frac{s+q}{1-s^2}.
\label{eq: h_q}\end{equation}

We can simplify $H_{\pm}(q,v)$ as follows.
Starting with
\begin{equation}
    h_q(s_1(q,v))=\log\frac{\left(v+1\right)\left(vq+y(q,v)\right)}{2\left(qy(q,v)-q^2+v+1\right)},
\label{eq: h_q(s1)}\end{equation}
we multiply the numerator and the denominator by $qy(q,v)+q^2-v-1$, which yields
\begin{equation}
    \left(qy(q,v)-q^2+v+1\right)\left(qy(q,v)+q^2-v-1\right)=\left(q^2-1\right)\left(v+1\right)^2,
\end{equation}
and
\begin{equation}
    \left(vq+y(q,v)\right)\left(qy(q,v)+q^2-v-1\right)=\left(q^2-1\right)\left(v+1\right)\left(y(q,v)+q\right).
\end{equation}
We obtain
\begin{equation}
    h_q(s_1(q,v))=\log\frac{y(q,v)+q}{2},
\end{equation}
which is clearly an even function.

As
\begin{equation}
    h_q(s_2(q,v))=\log\frac{v-y(q,v)}{2\left((v-q)y(q,v)-(v-q)^2-vq+1\right)},
\label{eq: 1-s_2^2}\end{equation}
we multiply by the factor $(v-q)y(q,v)+(v-q)^2+vq-1$, which gives
\begin{equation}
    \left(1-s_2^2(q,v)\right)\left((v-q)y(q,v)+(v-q)^2+vq-1\right)=2\left(1-q^2\right)\left(v^2-1\right),
\end{equation}
and
\begin{equation}
    \left(v-y(q,v)\right)\left((v-q)y(q,v)+(v-q)^2+vq-1\right)=\left(1-q^2\right)\left(y(q,v)-q\right).
\end{equation}
Therefore,
\begin{equation}
    h_q(s_2(q,v))=\log\frac{y(q,v)-q}{2\left(v^2-1\right)},
\end{equation}
which is an even function as well.

Now
\begin{equation}
\begin{aligned}
    h_q(s_1(q,v))-h_q(s_2(q,v))&=\log\frac{\left(y(q,v)+q\right)\left(v^2-1\right)}{y(q,v)-q} \\
    &=2\log(q+\sqrt{v^2+q^2-1})
\end{aligned}
\end{equation}
and
\begin{equation}
    h_q(-q\pm\sqrt{q^2-1})=\log\frac{q\pm\sqrt{q^2-1}}{2},
\end{equation}
so we can rewrite
\begin{equation}
    H_+(q,v)=\begin{cases}
        0 & \text{if } q<-1 \\
        0 & \text{if } |q|<1 \text{ and } v<\sqrt{2-2q}\\
        2\log(q+\sqrt{v^2+q^2-1}) & \text{if } |q|<1 \text{ and } v>\sqrt{2-2q}\\
        \log((q+\sqrt{v^2+q^2-1})(q-\sqrt{q^2-1})) & \text{if } q>1,
    \end{cases}
\end{equation}
and
\begin{equation}
    \!H_-(q,v)=\begin{cases}
        0 & \text{if } q<-1 \text{ and } |v|<\sqrt{2-2q} \\
        2\log(q+\sqrt{v^2+q^2-1}) & \text{if } q<-1 \text{ and } |v|>\sqrt{2-2q}\\
        0 & \text{if } |q|<1 \text{ and } v>-\sqrt{2-2q} \\
        2\log(q+\sqrt{v^2+q^2-1}) & \text{if } |q|<1 \text{ and } v<-\sqrt{2-2q}\\
        \log((q+\sqrt{v^2+q^2-1})(q-\sqrt{q^2-1})) & \text{if } q>1.
    \end{cases}
  \end{equation}

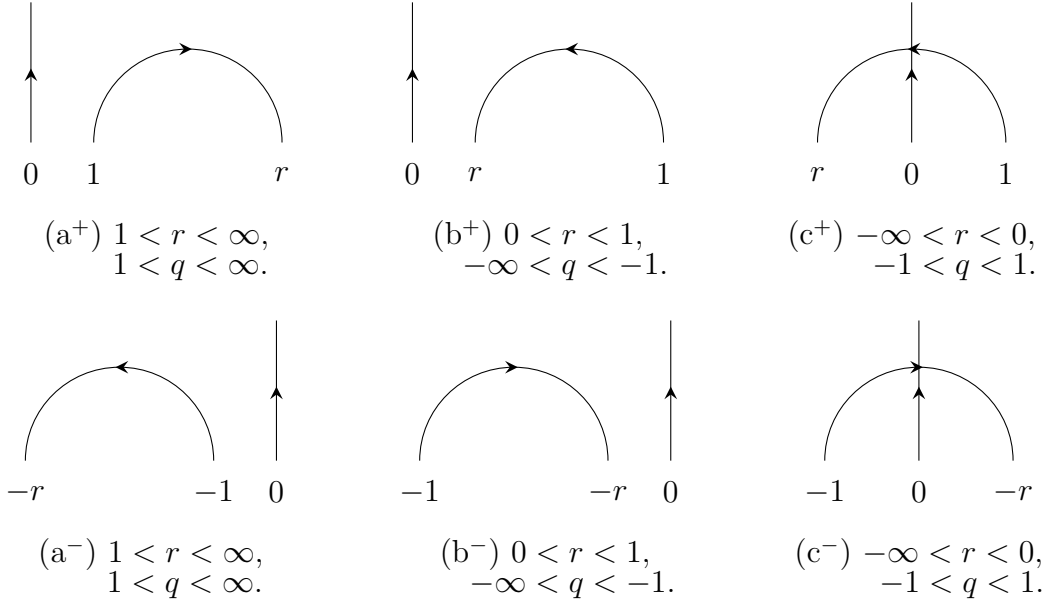
\begin{figure}[htb]
\centering
\subfigure{
\begin{tikzpicture}
    \draw (0em,0em) -- (0em,4.5em) node[scale=1.5,sloped,pos=0.5,allow upside down]{\arrowIn};
    \draw (2em,0em) arc (180:0:3em) node[scale=1.5,sloped,pos=0.5,allow upside down]{\arrowIn};
    \node at (0em,-1em) {$0$};
    \node at (2em,-1em) {$1$};
    \node at (8em,-1em) {$r$};
    \node at (4em,-3em) {(a$^+$) $1<r<\infty,$};
    \node at (4em,-4em) {$\qquad\, 1<q<\infty.$};
\end{tikzpicture}
}
\qquad
\subfigure{
\begin{tikzpicture}
    \draw (0em,0em) -- (0em,4.5em) node[scale=1.5,sloped,pos=0.5,allow upside down]{\arrowIn};
    \draw (8em,0em) arc (0:180:3em) node[scale=1.5,sloped,pos=0.5,allow upside down]{\arrowIn};
    \node at (0em,-1em) {$0$};
    \node at (2em,-1em) {$r$};
    \node at (8em,-1em) {$1$};
    \node at (4em,-3em) {(b$^+$) $0<r<1,$};
    \node at (4em,-4em) {$\quad\;\;\, -\infty<q<-1.$};
\end{tikzpicture}
}
\qquad
\subfigure{
\begin{tikzpicture}
    \draw (0em,0em) -- (0em,4.5em) node[scale=1.5,sloped,pos=0.5,allow upside down]{\arrowIn};
    \draw (3em,0em) arc (0:180:3em) node[scale=1.5,sloped,pos=0.5,allow upside down]{\arrowIn};
    \node at (0em,-1em) {$0$};
    \node at (-3em,-1em) {$r$};
    \node at (3em,-1em) {$1$};
    \node at (0em,-3em) {(c$^+$) $-\infty<r<0,$};
    \node at (0em,-4em) {$\qquad\;\;\, -1<q<1.$};
\end{tikzpicture}
}

\subfigure{
\begin{tikzpicture}
    \draw (0em,0em) -- (0em,4.5em) node[scale=1.5,sloped,pos=0.5,allow upside down]{\arrowIn};
    \draw (-2em,0em) arc (0:180:3em) node[scale=1.5,sloped,pos=0.5,allow upside down]{\arrowIn};
    \node at (0em,-1em) {$0$};
    \node at (-2em,-1em) {$-1$};
    \node at (-8em,-1em) {$-r$};
    \node at (-4em,-3em) {(a$^-$) $1<r<\infty,$};
    \node at (-4em,-4em) {$\qquad\, 1<q<\infty.$};
\end{tikzpicture}
}
\qquad
\subfigure{
\begin{tikzpicture}
    \draw (0em,0em) -- (0em,4.5em) node[scale=1.5,sloped,pos=0.5,allow upside down]{\arrowIn};
    \draw (-8em,0em) arc (180:0:3em) node[scale=1.5,sloped,pos=0.5,allow upside down]{\arrowIn};
    \node at (0em,-1em) {$0$};
    \node at (-2em,-1em) {$-r$};
    \node at (-8em,-1em) {$-1$};
    \node at (-4em,-3em) {(b$^-$) $0<r<1$,};
    \node at (-4em,-4em) {$\quad\;\;\, -\infty<q<-1.$};
\end{tikzpicture}
}
\qquad
\subfigure{
\begin{tikzpicture}
    \draw (0em,0em) -- (0em,4.5em) node[scale=1.5,sloped,pos=0.5,allow upside down]{\arrowIn};
    \draw (-3em,0em) arc (180:0:3em) node[scale=1.5,sloped,pos=0.5,allow upside down]{\arrowIn};
    \node at (0em,-1em) {$0$};
    \node at (-3em,-1em) {$-1$};
    \node at (3em,-1em) {$-r$};
    \node at (0em,-3em) {(c$^-$) $-\infty<r<0,$};
    \node at (0em,-4em) {$\qquad\;\;\, -1<q<1.$};
\end{tikzpicture}
}
\caption{The six standard configurations}
\label{fig: 6 standard diagrams}\end{figure}

With these expressions, it is clear to see that $H_\pm(q,v)$ are even functions of $v$ for a fixed $|q|>1$.
However, for $|q| < 1$ these are not even functions of $v$.
That the sum \eqref{eq: W(v)} is even (which of course follows immediately from the definition of the pair correlation) can be made manifest as follows:
Any pair of geodesics in $\mathbb{H}$ that do not share an endpoint in $\partial \mathbb{H}$ can be brought by an element of $\SL(2, \mathbb{R})$ into one of the six standard forms illustrated in Figure \ref{fig: 6 standard diagrams}.
The cases with $|q| < 1$ are (c$^+$) and (c$^-$), the sign corresponding the whether the pair contributes to $H_+$ or $H_-$, see \cite{misc:Marklof-Welsh}.
The evenness of the sum \eqref{eq: W(v)} then follows from the fact that the terms corresponding to (c$^+$) are in $q$-preserving bijection with those corresponding to (c$^-$).

Such a bijection is given by $\Gamma_{\bm{c}_{l_1}} \gamma \Gamma_{\bm{c}_{l_2}} \mapsto \Gamma_{\bm{c}_{l_2}} \gamma^{-1} \Gamma_{\bm{c}_{l_1}}$.
Indeed, if there is $g_1 \in \SL(2, \mathbb{R})$ so that the geodesics $g_1\bm{c}_{l_1}$ and $g_2\gamma \bm{c}_{l_2}$ are as in (c$^+$) with $g_1 \bm{c}_{l_1}$ vertical, then $ g_1 \gamma$ brings the geodesics $\bm{c}_{l_2}$ and $\gamma^{-1} \bm{c}_{l_1}$ to the form (c$^+$), however with $g_2 \gamma^{-1} \bm{c}_{l_1}$ vertical.
The geodesics can then be brought to the form (c$^-$), with $\bm{c}_{l_2}$ mapping to vertical, as required, by an application of the transformation $x \mapsto \frac{-rx + r}{x - r}$, which can be seen to preserve the cross-ratio $r$ and hence $q$.

\section{Negative \texorpdfstring{$D$}{D}}
\label{sec:Dnegative}

In this final section, we briefly sketch analogous results in the case that $D$ is negative.
The negative setting is simpler in the sense that we consider $\Gamma$-orbits of points in $\mathbb{H}$ instead of the tops of geodesics, and their statistical distribution has been thoroughly studied.
However, algebraically the $D < 0$ setting is almost identical with the following modification: instead of embedding $\mathbb{Q}(\sqrt{D})$ into a two-dimensional real subspace of $\mathbb{C}^2$ via $(\xi, \overline{\xi})$ as in the real case, we embed $\mathbb{Q}(\sqrt{D})$ inside $\mathbb{R}^2$ by
\begin{equation*}
    a+b\sqrt{D}\mapsto (a,b\sqrt{|D|}).
\end{equation*}
The first embedding is transformed into the second by the matrix $P=\begin{pmatrix} -\frac{\mathrm{i}}{2} & \frac{1}{2} \\ \frac{\mathrm{i}}{2} & \frac{1}{2} \end{pmatrix}$.
Since
\begin{equation}
    \begin{pmatrix}
        \sqrt{D} & -\sqrt{D} \\
        1 & 1
    \end{pmatrix}P
    =
    \begin{pmatrix}
        \sqrt{|D|} & 0 \\
        0 & 1
    \end{pmatrix}
\end{equation}
and
\begin{equation}
    P^{-1}
    \begin{pmatrix}
        \xi & 0\\
        0 & \overline{\xi}
    \end{pmatrix}P
    =
    \begin{pmatrix}
        \Re\xi & -\Im\xi \\
        \Im\xi & \Re\xi
    \end{pmatrix},
\end{equation}
then rewriting \eqref{eq: gamma B_k}, we have
\begin{equation}
    \begin{pmatrix}
        1 & \mu \\
        0 & m
    \end{pmatrix}
    \begin{pmatrix}
        \sqrt{|D|} & 0 \\
        0 & 1
    \end{pmatrix}
    =
    \gamma \widetilde{\mathfrak{B}}_k
    \begin{pmatrix}
        \Re\xi & -\Im\xi \\
        \Im\xi & \Re\xi
    \end{pmatrix},
\label{eq: D<0 I}\end{equation}
where
\begin{equation}
    \widetilde{\mathfrak{B}}_k=\mathfrak{B}_kP
    =
    \begin{pmatrix}
        \Im\beta_{k1} & \Re\beta_{k1} \\
        \Im\beta_{k2} & \Re\beta_{k2} 
    \end{pmatrix}.
\end{equation}
We note that the matrix on the right of \eqref{eq: D<0 I} stabilizes the point $\mathrm{i} \in \mathbb{H}$, as opposed to stabilizing the geodesic $\{\mathrm{i}y: y > 0 \} \subset \mathbb{H}$ in the real setting.  

  The left side of \eqref{eq: D<0 I} yields
\begin{equation}
    \begin{pmatrix}
        1 & \mu \\
        0 & m
    \end{pmatrix}
    \begin{pmatrix}
        \sqrt{|D|} & 0 \\
        0 & 1
    \end{pmatrix}\mathrm{i}
    =
    \frac{\mu}{m}+\mathrm{i}\frac{\sqrt{|D|}}{m},
\end{equation}
and the right side of \eqref{eq: D<0 I} yields
\begin{equation}
    \gamma\begin{pmatrix}
        \Im\beta_{k1} & \Re\beta_{k1} \\
        \Im\beta_{k2} & \Re\beta_{k2} 
    \end{pmatrix}
    \begin{pmatrix}
        \Re\xi & -\Im\xi \\
        \Im\xi & \Re\xi
    \end{pmatrix}\mathrm{i}
    =\gamma\frac{\xi\beta_{k1}}{\xi\beta_{k2}}=\gamma z_{I_k}, \;\text{where}\; z_{I_k}=\frac{\beta_{k1}}{\beta_{k2}}.
\end{equation}
Therefore, $\gamma z_{I_k}=\frac{\mu}{m}+\mathrm{i}\frac{\sqrt{|D|}}{m}$. With a similar construction as in the positive $D$ case, we have the following analogue of Theorem \ref{thm: congruence subgroups I}. 
\begin{thm}
    Let $D<0$ be a square-free integer satisfying $D\equiv 1 \pmod 4$ and $h_1$ be the class number of $\mathbb{Z}[\sqrt{D}]$. Given $n>0$ and $\nu \bmod n$ with $\nu^2\equiv D \pmod n$, there exists a finite set of points $\{z_{I_1},\cdots,z_{I_{h_1}}\}$ in $\mathbb{H}$ with the following properties:
    \begin{enumerate}[(i)]
        \item For $m\in\mathbb{N}$, let $\mu \bmod m$ satisfy $\mu^2\equiv D \pmod m$ such that $m\equiv 0\pmod n$ and $\mu\equiv\nu\pmod n$. If either $m$ or $\frac{D-\mu^2}{m}$ is odd, then there exists a unique $k$ and a double coset $\Gamma_\infty\gamma\Gamma_{1;k}\in\Gamma_\infty\backslash\Gamma_0(n)\slash\Gamma_{1;k}$ such that 
        \begin{equation}
            \gamma z_{I_k}\equiv\frac{\mu}{m}+\mathrm{i}\frac{\sqrt{|D|}}{m} \pmod {\Gamma_\infty}.
        \label{eq:D<0 gamma1}\end{equation}
        \item Conversely, given $k$ and a double coset $\Gamma_\infty\gamma\Gamma_{1;k}\in\Gamma_\infty\backslash\Gamma_0(n)\slash\Gamma_{1;k}$, there exists a unique positive integer $m$ and a residue class $\mu \bmod m$ satisfying \eqref{eq:D<0 gamma1} and either $m$ or $\frac{D-\mu^2}{m}$ being odd.
    \end{enumerate}
  \end{thm}

In the same way, we obtain the analogue of Theorem \ref{thm: congruence subgroups J}.

\begin{thm}
    Let $D<0$ be a square-free integer satisfying $D\equiv 1 \pmod 4$ and $h_2$ be the class number of $\mathbb{Z}[\frac{1+\sqrt{D}}{2}]$. Given $n>0$ and $\nu \bmod n$ with $\nu^2\equiv D \pmod n$, there exists a finite set of points $\{z_{J_1},\cdots,z_{J_{h_2}}\}$ in $\mathbb{H}$ with the following properties:
    \begin{enumerate}[(i)]
        \item For even $m\in\mathbb{N}$, let $\mu \bmod m$ satisfy $\mu^2\equiv D \pmod m$ and $\frac{D-\mu^2}{m}$ even such that $m\equiv 0\pmod n$ and $\mu\equiv\nu\pmod n$. If $n$ is odd (resp. $n$ is even), then there exists a unique $l$ and a double coset $\Gamma_\infty\gamma\Gamma_{2;l}\in\Gamma_\infty\backslash\Gamma_0(n)\slash\Gamma_{2;l}$ (resp. $\Gamma_\infty\backslash\Gamma_0(\frac{n}{2})\slash\Gamma_{2;l}$) such that 
        \begin{equation}
            \gamma z_{J_l}\equiv\frac{\mu}{m}+\mathrm{i}\frac{\sqrt{D}}{m} \pmod {\Gamma_\infty}.
        \label{eq:D<0 gamma2}\end{equation}
        \item Conversely, given $l$ and a double coset $\Gamma_\infty\gamma\Gamma_{J;l}\in\Gamma_\infty\backslash\Gamma_0(n)\slash\Gamma_{2;l}$ (resp. ${\Gamma_\infty\backslash\Gamma_0(\frac{n}{2})\slash\Gamma_{2;l}}$), there exists a unique even positive integer $m$ and a residue class $\mu \bmod m$ satisfying \eqref{eq:D<0 gamma2} and $\frac{D-\mu^2}{m}$ being even.
    \end{enumerate}
\end{thm}
Combining these two theorems, we obtain the analogue of Theorem \ref{thm: congruence subgroups}.
\begin{thm}
    Let $D<0$ be a square-free integer satisfying $D\equiv 1 \pmod 4$. Given $n>0$ and $\nu \bmod n$ with $\nu^2\equiv D \pmod n$, there exists a finite set of points $\{z_1,\cdots,z_{h}\}$ in $\mathbb{H}$ with the following properties:
    \begin{enumerate}[(i)]
        \item For $m\in\mathbb{N}$, let $\mu \bmod m$ satisfy $\mu^2\equiv D \pmod m$ such that $m\equiv 0\pmod n$ and $\mu\equiv\nu\pmod n$. Then there exists a unique $k$ and a double coset $\Gamma_\infty\gamma\Gamma_k\in\Gamma_\infty\backslash\Gamma_0(n)\slash\Gamma_k$ such that 
        \begin{equation}
            \gamma z_k\equiv\frac{\mu}{m}+\mathrm{i}\frac{\sqrt{D}}{m} \pmod {\Gamma_\infty}.
        \label{eq:D<0 gamma1&2}\end{equation}
        \item Conversely, given $k$ and a double coset $\Gamma_\infty\gamma\Gamma_k\in\Gamma_\infty\backslash\Gamma_0(n)\slash\Gamma_k$, there exists a unique positive integer $m$ and a residue class $\mu \bmod m$ satisfying \eqref{eq:D<0 gamma1&2}.
    \end{enumerate}
\end{thm}

Now, to obtain the analogues of Theorems \ref{thm: pair correlation density} and \ref{thm: random point process}, we can apply the results of \cite{jr:Marklof-Vinogradov} instead of those in \cite{misc:Marklof-Welsh}.
The paper \cite{jr:Marklof-Vinogradov} studies the distribution of real parts of an orbit of a point in $\mathbb{H}$ (and higher-dimensional hyperbolic space).
This is referred to as the "cuspidal" observer case of the distribution of angles in hyperbolic lattices, which has been studied in \cite{jr:Boca-Pasol-Popa,jr:Boca-Popa,jr:Lutsko,jr:Risager-Rudnick,jr:Risager-Sodergren} and others. 
The pair correlation in this setting is computed in \cite{jr:Kelmer-Kontorovich} and \cite{jr:Marklof-Vinogradov}, giving a form similar to the one obtained here and in \cite{misc:Marklof-Welsh}.


\bibliographystyle{plain}
\bibliography{refs}

\begin{thebibliography}{10}

\bibitem{jr:Boca-Pasol-Popa}
Florin~P. Boca, Vicentiu Pasol, Alexandru~A. Popa, and Alexandru Zaharescu.
\newblock {Pair correlation of angles between reciprocal geodesics on the
  modular surface}.
\newblock {\em Algebra Number Theory}, 8(4):999--1035, 2014.

\bibitem{jr:Boca-Popa}
Florin~P. Boca, Alexandru~A. Popa, and Alexandru Zaharescu.
\newblock {Pair correlation of hyperbolic lattice angles}.
\newblock {\em International Journal of Number Theory}, 10(8):1955--1989, 2014.

\bibitem{jr:Bykovskii}
Victor~A. Bykovskii.
\newblock {Spectral decompositions of certain automorphic functions and their
  number-theoretic applications}.
\newblock {\em Journal of Soviet Mathematics}, 36(1):8--21, 1987.

\bibitem{jr:DFI}
W.~Duke, J.~Friedlander, and H.~Iwaniec.
\newblock {Equidistribution of roots of a quadratic congruence to prime
  moduli}.
\newblock {\em The Annals of Mathematics}, 141(2):423--441, 1995.

\bibitem{jr:Hejhal}
Dennis~A. Hejhal.
\newblock {Roots of quadratic congruences and eigenvalues of the non-Euclidean
  Laplacian}.
\newblock {\em The Selberg trace formula and related topics, Contemporary
  Mathematics}, 53:227--339, 1986.

\bibitem{jr:Hooley1}
Christopher Hooley.
\newblock {On the number of divisors of quadratic polynomials}.
\newblock {\em Acta Mathematica}, 110:97--114, 1963.

\bibitem{jr:Iwaniec}
Henryk Iwaniec.
\newblock {Almost-primes represented by quadratic polynomials}.
\newblock {\em Inventiones Mathematicae}, 47(2):171--188, 1978.

\bibitem{jr:Kelmer-Kontorovich}
Dubi Kelmer and Alex Kontorovich.
\newblock {On the pair correlation density for hyperbolic angles}.
\newblock {\em Duke Mathematical Journal}, 164(3):473--509, 2015.

\bibitem{jr:Lutsko}
Christopher Lutsko.
\newblock {Directions in orbits of geometrically finite hyperbolic subgroups}.
\newblock {\em Mathematical Proceedings of the Cambridge Philosophical
  Society}, 171:277--316, 2020.

\bibitem{jr:Marklof-Vinogradov}
Jens Marklof and Ilya Vinogradov.
\newblock {Directions in hyperbolic lattices}.
\newblock {\em Journal f\"ur die reine und angewandte Mathematik (Crelles
  Journal)}, 2018(740):161--186, 2018.

\bibitem{misc:Marklof-Welsh}
Jens Marklof and Matthew Welsh.
\newblock {Fine-scale distribution of roots of quadratic congruences}, 2021(See
  \href{https://arxiv.org/pdf/2105.02854.pdf}{arXiv:2105.02854}).
\newblock Duke Mathematical Journal to appear.

\bibitem{jr:Risager-Rudnick}
Morten~S. Risager and Ze\'ev Rudnick.
\newblock {On the statistics of the minimal solution of a linear Diophantine
  equation and uniform distribution of the real part of orbits in hyperbolic
  spaces}.
\newblock {\em Spectral analysis in geometry and number theory, Contemporary
  Mathematics}, 484:187--194, 2009.

\bibitem{jr:Risager-Sodergren}
Morten~S. Risager and Anders S\"odergren.
\newblock {Angles in hyperbolic lattices: the pair correlation density}.
\newblock {\em Transactions of the American Mathematical Society},
  369(4):2807--2841, 2017.

\end{thebibliography}

\end{document}